\documentclass{amsart} 
\usepackage{amscd}
\usepackage{amsmath,empheq}
\usepackage{amsfonts}
\usepackage{amssymb}
\usepackage{mathrsfs}
\newtheorem{theorem}{Theorem}

\newtheorem{prop}[theorem]{Proposition}
\newtheorem{remark}{Remark}

\newtheorem{claim}{Claim}

\newenvironment{proof-sketch}{\noindent{\bf Sketch of Proof}\hspace*{1em}}{\qed\bigskip}

\newcommand{\RR}{\mathbb R}
\newcommand{\NN}{\mathbb N}

\newcommand{\ZZ}{\mathbb Z}

\renewcommand{\leq}{\leqslant}

\renewcommand{\geq}{\geqslant}
\begin{document}
\title[Nonlinear Dirichlet problems with unilateral growth]{Nonlinear Dirichlet problems with unilateral growth on the reaction}
\author[N.S. Papageorgiou]{Nikolaos S. Papageorgiou}
\address[N.S. Papageorgiou]{National Technical University, Department of Mathematics,
				Zografou Campus, Athens 15780, Greece \& Institute of Mathematics, Physics and Mechanics, 1000 Ljubljana, Slovenia}
\email{\tt npapg@math.ntua.gr}
\author[V.D. R\u{a}dulescu]{Vicen\c{t}iu D. R\u{a}dulescu}
\address[V.D. R\u{a}dulescu]{Faculty of Applied Mathematics, AGH University of Science and Technology, al. Mickiewicza 30, 30-059 Krak\'ow, Poland \& Institute of Mathematics, Physics and Mechanics, 1000 Ljubljana, Slovenia \& Institute of Mathematics ``Simion Stoilow" of the Romanian Academy, P.O. Box 1-764, 014700 Bucharest, Romania}
\email{\tt vicentiu.radulescu@imar.ro}
\author[D.D. Repov\v{s}]{Du\v{s}an D. Repov\v{s}}
\address[D.D. Repov\v{s}]{Faculty of Education and Faculty of Mathematics and Physics, University of Ljubljana \& Institute of Mathematics, Physics and Mechanics, 1000 Ljubljana, Slovenia}
\email{\tt dusan.repovs@guest.arnes.si}
\bigskip
\begin{abstract}
We consider a nonlinear Dirichlet problem driven by the $p$-Laplace differential operator with a reaction which has a subcritical growth restriction only from above. We prove two multiplicity theorems producing three nontrivial solutions, two of constant sign and the third nodal. The two multiplicity theorems differ on the geometry near the origin. In the semilinear case (that is, $p=2$), using Morse theory (critical groups), we produce a second nodal solution for a total of four nontrivial solutions. As an illustration, we show that our results incorporate and  significantly extend the multiplicity results existing for a class of parametric, coercive Dirichlet problems.
\end{abstract}
\maketitle

\section{Introduction}

Let $\Omega\subseteq\RR^N$ be a bounded domain with a $C^2$-boundary $\partial\Omega$. In this paper we study the following nonlinear Dirichlet problem
\begin{equation}\label{eq1}\left\{\begin{array}{lll}
&\displaystyle	-\Delta_pu(z)=f(z,u(z))&\quad \mbox{in}\ \Omega\\
&\displaystyle u=0&\quad\mbox{on}\ \partial\Omega.\end{array}\right.
\end{equation}

Here, $\Delta_p$ denotes the $p$-Laplace differential operator defined by
$$\Delta_pu=\mbox{div}\, (|Du|^{p-2}Du)\ \mbox{for all}\ u\in W^{1,p}_{0}(\Omega),\ 1<p<\infty.$$

Usually such problems are examined under the assumption that the reaction $f(z,\cdot)$ exhibits subcritical growth from above and below. In contrast, we assume here that $f(z,\cdot)$ is subcritical only from above, while from below no growth restriction is imposed on $f(z,\cdot)$. In this setting, we prove a multiplicity theorem producing at least three nontrivial solutions, two of constant sign (one positive and one negative) and the third nodal (that is, sign-changing). Our multiplicity result compares with those proved by Liu \& Liu \cite{14}, Liu \cite{15}, and Papageorgiou \& Papageorgiou \cite{18} who proved three solutions theorems for certain classes of coercive $p$-Laplacian equations. We also refer to Papageorgiou, R\u adulescu \$ Repov\v{s} \cite{prr1,prr2}
for multiplicity properties in the context of Robin problems with superlinear reaction and  super-diffusive mixed problems.

In all the aforementioned works, the reaction has bilateral subcritical growth and no nodal solutions are produced. In addition, in the present work, for the semilinear problem $(p=2)$, using Morse theory (critical groups), we produce a second nodal solution, for a total of four nontrivial solutions. Finally, we mention the works of Villegas \cite{21} and Filippakis, Gasinski \& Papageorgiou \cite{7} who proved existence theorems for unilaterally restricted scalar problems (that is, $N=1$). Villegas \cite{21} studied semilinear (that is, $p=2$) Neumann problems and Filippakis, Gasinski \& Papageorgiou \cite{7} considered nonlinear (that is, $1<p<\infty$) periodic with a nonsmooth potential.

\section{Mathematical background}

Let $X$ be a Banach space, $X^*$ its topological dual, and let $\left\langle \cdot,\cdot\right\rangle$  denote the duality brackets for the pair $(X^*,X)$. We say that a function $\varphi\in C^1(X)$ satisfies the {\it Palais-Smale condition} ({\it $PS$-condition}, for short), if the following property holds:
\begin{center}
``Every sequence $\{u_n\}_{n\geq 1}\subseteq X$ such that $\{\varphi(u_n)\}_{n\geq 1}\subseteq X$ is bounded and
	$$\varphi'(u_n)\rightarrow 0\ \mbox{in}\ X^*\ \mbox{as}\ n\rightarrow\infty,$$
	admits a strongly convergent subsequence."
\end{center}

This is a compactness-type condition on the functional $\varphi$, which leads to a deformation theorem, from which one can derive the minimax theory of the critical values of $\varphi$. A basic result in this theory is the so-called ``mountain pass theorem", due to Ambrosetti \& Rabinowitz \cite{4}.
\begin{theorem}\label{th1}
	Assume that $X$ is a Banach space, $\varphi\in C^1(X)$ and satisfies the $PS$-condition, $u_0,u_1\in X,\ ||u_1-u_0||>\rho>0$,
	$$\max\{\varphi(u_0),\varphi(u_1)\}<\inf\left\{\varphi(u):||u-u_0||=\rho\right\}=m_{\rho}$$
	and $c=\inf\limits_{\gamma\in \Gamma}\max\limits_{0\leq t\leq 1}\varphi(\gamma(t))$, where $\Gamma=\{\gamma\in C([0,1],X):\gamma(0)=u_0,\gamma(1)=u_1\}$. Then $c\geq m_{\rho}$ and $c$ is a critical value of $\varphi$.
\end{theorem}

In the study of problem (\ref{eq1}) we will use the Sobolev space $W^{1,p}_{0}(\Omega)$ (when $p=2$, we will write $H^{1}_{0}(\Omega)$) and the ordered Banach space $C^{1}_{0}(\overline{\Omega})=\{u\in C^1(\overline{\Omega}):u|_{\partial\Omega}=0\}$, with the order cone
$$C_+=\{u\in C^1_0(\overline{\Omega}):u(z)\geq 0\ \mbox{for all}\ z\in \overline{\Omega}\}.$$

This cone has a nonempty interior given by
$${\rm int}\,C_+=\{u\in C_+:u(z)>0\ \mbox{for all}\ z\in\Omega,\ \frac{\partial u}{\partial n}<0\ \mbox{on}\ \partial\Omega\}.$$

Here  we denote the outward unit normal on $\partial\Omega$ by $n(\cdot)$.

Let $f_0:\Omega\times\RR\rightarrow\RR$ be a Carath\'eodory function such that
$$|f_0(z,x)|\leq a(z)(1+|x|^{r-1})\ \mbox{for almost all}\ z\in\Omega,\ \mbox{and all}\ x\in \RR,$$
with $a\in L^{\infty}(\Omega)$ and $$1<r<p^*=\left\{\begin{array}{ll}
	\frac{Np}{N-p}&\mbox{if}\ p<N\\
	+\infty&\mbox{if}\ p\geq N
\end{array}\right.\quad \mbox{(the critical Sobolev exponent)}.$$ We set $F_0(z,x)=\int^{x}_{0}f_0(z,s)ds$ and consider the $C^1$-functional $\varphi_0:W^{1,p}_{0}(\Omega)\rightarrow\RR$ defined by
$$\varphi_0(u)=\frac{1}{p}||Du||^p_p-\int_{\Omega}F_0(z,u(z))dz\ \mbox{for all}\ u\in W^{1,p}_{0}(\Omega).$$

From Garcia Azorero, Manfredi \& Peral Alonso \cite{9}, we recall the following result.
\begin{prop}\label{prop2}
	Assume that $u_0\in W^{1,p}_{0}(\Omega)$ is a local $C^1_0(\overline{\Omega})$-minimizer of $\varphi_0$, that is, there exists $\rho_0>0$ such that
	$$\varphi_0(u_0)\leq\varphi_0(u_0+h)\ \mbox{for all}\ h\in C^1_0(\overline{\Omega})\ \mbox{with}\ ||h||_{C^1_0(\overline{\Omega})}\leq\rho_0.$$
Then $u_0\in C^{1,\alpha}_{0}(\overline{\Omega})$ for some $\alpha\in(0,1)$ and $u_0$ is also a local $W^{1,p}_{0}(\Omega)$-minimizer of $\varphi_0$, that is, there exists $\rho_1>0$ such that
	$$\varphi_0(u_0)\leq \varphi_0(u_0+h)\ \mbox{for all}\ h\in W^{1,p}_{0}(\Omega)\ \mbox{with}\ ||h||\leq\rho_1.$$
\end{prop}

Hereafter,  we denote the norm of the Sobolev space $W^{1,p}_{0}(\Omega)$ by $||\cdot||$. By the Poincar\'e inequality we have
$$||u||=||Du||_{p}\ \mbox{for all}\ u\in W^{1,p}_{0}(\Omega).$$

We will also use some basic facts about the spectrum of $(-\Delta_p,W^{1,p}_{0}(\Omega))$. So, we consider the following nonlinear eigenvalue problem
\begin{equation}\label{eq2}
	-\Delta_pu(z)=\hat{\lambda}|u(z)|^{p-2}u(z)\ \mbox{in}\ \Omega,\ u|_{\partial\Omega}=0.
\end{equation}

We say that $\hat{\lambda}$ is an {\it eigenvalue} of $(-\Delta_p,W^{1,p}_{0}(\Omega))$, if the problem (\ref{eq2}) admits a nontrivial solution $\hat{u}\in W^{1,p}_{0}(\Omega)$, which is an eigenfunction corresponding to the eigenvalue $\hat{\lambda}$. We know that there is a smallest eigenvalue $\hat{\lambda}_1>0$, which is simple, isolated and admits the following variational characterization:
\begin{equation}\label{eq3}
	\hat{\lambda}_1=\inf\left\{\frac{||Du||^p_p}{||u||^p_p}:u\in W^{1,p}_{0}(\Omega),u\neq 0\right\}.
\end{equation}

The infimum in (\ref{eq3}) is realized on the corresponding one-dimensional eigenspace (recall that $\hat{\lambda}_1>0$ is simple). It is clear from (\ref{eq3}) that the elements of this eigenspace do not change sign. Let $\hat{u}_1$ be the $L^p$-normalized, positive eigenfunction corresponding to $\hat{\lambda}_1>0$. From the nonlinear regularity theory and the nonlinear maximum principle (see, for example, Gasinski \& Papageorgiou \cite[pp. 737-738]{10}), we have $\hat{u}_1\in {\rm int}\,C_+$. From the Ljusternik-Schnirelmann minimax scheme, we can obtain a whole strictly increasing sequence $\{\hat{\lambda}_k\}_{k\geq 1}$ of eigenvalues such that $\hat{\lambda}_k\rightarrow+\infty$. We do not know if this sequence exhausts the spectrum of $(-\Delta_p,W^{1,p}_{0}(\Omega))$. This is the case if $p=2$ (linear eigenvalue problem) or $N=1$ (scalar eigenvalue problem). Since $\hat{\lambda}_1$ is isolated, the second eigenvalue $\hat{\lambda}^*_2>\hat{\lambda}_1$ is well-defined by
$$\hat{\lambda}^*_2=\inf\{\hat{\lambda}:\hat{\lambda}>\hat{\lambda}_1,\hat{\lambda}\ \mbox{is an eigenvalue of}\ (-\Delta_p,W^{1,p}_{0}(\Omega))\}.$$

We know that $\hat{\lambda}^*_2=\hat{\lambda}_2$, that is, the second eigenvalue and the second Ljusternik-Schnirelmann eigenvalue coincide. For $\hat{\lambda}_2$ we have the following minimax characterization due to Cuesta, de Figueiredo \& Gossez \cite{6}.
\begin{prop}\label{prop3} We have
	$$\hat{\lambda}_2=\inf\limits_{\hat{\gamma}\in\hat{\Gamma}}\max\limits_{-1\leq t\leq 1}||D\hat{\gamma}(t)||^p_p,$$ where $$\hat{\Gamma}=\{\hat{\gamma}\in C([-1,1],M):\hat{\gamma}(-1)=-\hat{u}_1,\hat{\gamma}(1)=\hat{u}_1\}$$ and
	$$M=W^{1,p}_{0}(\Omega)\cap \partial B^{L^p}_{1}\ \text{and}\ \partial B^{L^p}_{1}=\{u\in L^p(\Omega):||u||_p=1\}.$$
\end{prop}

As we already said, in the case $p=2$ (linear eigenvalue problem), the spectrum of $(-\Delta, H^1_0(\Omega))$ consists of a sequence $\{\hat{\lambda}_k\}_{k\geq 1}$ of eigenvalues such that $\hat{\lambda}_k\rightarrow+\infty$ as $k\rightarrow+\infty$. We denote the corresponding eigenspace by $E(\hat{\lambda}_k)$. We have
$$H^1_0(\Omega)=\overline{{\underset{{k\geq 1}}\oplus}E(\hat{\lambda}_k)}.$$

In this case, we have nice variational characterizations for all the eigenvalues. Namely, we have
\begin{equation}\label{eq4}
	\hat{\lambda}_1=\inf\left\{\frac{||Du||^2_2}{||u||^2_2}:u\in H^1_0(\Omega),u\neq 0\right\}\ \mbox{see (\ref{eq3})}
\end{equation}
and for $k\geq 2$
\begin{eqnarray}\label{eq5}
	\hat{\lambda}_k&=&\inf\left\{\frac{||Du||^2_2}{||u||^2_2}:u\in \widehat{H}_k=\overline{{\underset{{n\geq k}}\oplus}E(\hat{\lambda}_n)},\ u\neq 0\right\}\nonumber\\
	&=&\sup\left\{\frac{||Du||^2_2}{||u||^2_2}:u\in \bar{H}_k=\overset{k}{\underset{{n=1}}\oplus}E(\hat{\lambda}_n),\ u\neq 0\right\}.
\end{eqnarray}

 Both the infimum and the supremum in (\ref{eq5}) are realized on the corresponding eigenspace $E(\hat{\lambda}_k)$. Every such space has the so-called ``unique continuation property" (UCP for short), which means that if $u\in E(\hat{\lambda}_k)$ and $u$ vanishes on a set of positive measure, then $u\equiv 0$. Note that by standard regularity theory, $E(\hat{\lambda}_k)\subseteq C^1_0(\overline{\Omega})$ and $E(\hat{\lambda}_k)$ is finite-dimensional. Invoking (\ref{eq4}), (\ref{eq5}) and the UCP, we have the following property.
\begin{prop}\label{prop4}
	\begin{itemize}
		\item[(a)] If $\xi\in L^{\infty}(\Omega)$ with $\xi(z)\geq\hat{\lambda}_k$ for almost all $z\in\Omega$ with strict inequality on a set of positive measure, then
		$$||Du||^2_2-\int_{\Omega}\xi(z)u^2dz\leq -\hat{c}||u||^2\ \mbox{for all}\ u\in \bar{H}_k=\overset{k}{\underset{\mathrm{n=1}}\oplus}E(\hat{\lambda}_k).$$
		\item[(b)] If $\xi\in L^{\infty}(\Omega)$ with $\xi(z)\leq\hat{\lambda}_k$ for almost all $z\in\Omega$ with strict inequality on a set of positive measure, then there exists $\tilde{c}>0$ such that
		$$||Du||^2_2-\int_{\Omega}\xi(z)u^2dz\geq\tilde{c}||u||^2\ \mbox{for all}\ u\in \overline{{\underset{\mathrm{n\geq k}}\oplus}E(\hat{\lambda}_k)}.$$
	\end{itemize}
\end{prop}

In what follows, we denote by $$A:W^{1,p}_{0}(\Omega)\rightarrow W^{-1,p'}(\Omega)=W^{1,p}_{0}(\Omega)^*\quad\left(\frac{1}{p}+\frac{1}{p'}=1\right)$$  the nonlinear map corresponding to the $p$-Laplace differential operator and defined by
\begin{equation}\label{eq6}
	\left\langle A(u),v\right\rangle=\int_{\Omega}|Du|^{p-2}(Du,Dv)_{\RR^N}dz\ \mbox{for all}\ u,v\in W^{1,p}_{0}(\Omega).
\end{equation}

From Papageorgiou \& Kyritsi \cite[p. 314]{17}, we have:
\begin{prop}\label{prop5}
	The operator $A:W^{1,p}_{0}(\Omega)\rightarrow W^{-1,p'}(\Omega)$ defined by (\ref{eq6}) is continuous, strictly monotone (hence maximal monotone, too) and of type $(S)_+$, that is,
	$$``u_n\stackrel{w}{\rightarrow}u\ \mbox{in}\ W^{1,p}_{0}(\Omega)\ \mbox{and}\ \limsup\limits_{n\rightarrow\infty}\left\langle A(u_n),u_n-u\right\rangle\leq 0\Rightarrow u_n\rightarrow u\ \mbox{in}\ W^{1,p}_{0}(\Omega)".$$
\end{prop}
	
	As before, let $X$ be a Banach space, $\varphi\in C^1(X)$, and let $c\in \RR$. We introduce the following sets:
	$$K_{\varphi}=\{u\in X:\varphi'(u)=0\},\ K^{c}_{\varphi}=\{u\in K_{\varphi}:\varphi(u)=c\},\ \varphi^c=\{u\in X:\varphi(u)\leq c\}.$$
	
	Let $(Y_1,Y_2)$ be a topological pair such that $Y_2\subseteq Y_1\subseteq X$ and let $k$ be a positive integer.  We denote by $H_k(Y_1,Y_2)$ the $k$th-relative singular homology group of the topological pair $(Y_1,Y_2)$ with integer coefficients. The critical groups of $\varphi$ at an isolated $u\in K^{c}_{\varphi}$, are defined by
$$C_k(\varphi,u)=H_k(\varphi^c\cap U,\varphi^c\cap U\backslash\{u\})\ \mbox{for all}\ k,$$
with $U$ being a neighborhood of $u$ such that $K_{\varphi}\cap\varphi^c\cap U=\{u\}$. The excision property of singular homology, implies that the above definition of critical groups is independent of the choice of the neighborhood $U$ of $u$.

Suppose that $\varphi$ satisfies the $PS$-condition and $-\infty<\inf\varphi(K_{\varphi})$. Let $c<\inf\varphi(K_{\varphi})$. The critical groups of $\varphi$ at infinity are defined by
$$C_k(\varphi,\infty)=H_k(X,\varphi^c)\ \mbox{for all}\ k.$$

The second deformation theorem (see, for example, Gasinski \& Papageorgiou \cite[p. 628]{10}) implies that the above definition of critical groups at infinity is independent of the choice of the level $c<\inf\varphi(K_{\varphi})$.

Suppose that $\varphi\in C^1(X)$ satisfies the $PS$-condition and $K_{\varphi}$ is finite. We define
\begin{eqnarray*}
	&&M(t,u)=\sum_{k\geq 0}\mbox{rank}\, C_k(\varphi,u)t^k\ \mbox{for all}\ t\in\RR,\  u\in K_{\varphi},\\
	&&P(t,\infty)=\sum_{k\geq 0}\mbox{rank}\, C_k(\varphi,\infty)t^k\ \mbox{for all}\ t\in\RR.
\end{eqnarray*}	

The Morse relation says that
\begin{equation}\label{eq7}
	\sum_{u\in K_{\varphi}}M(t,u)=P(t,\infty)+(1+t)Q(t),
\end{equation}
where $Q(t)={\underset{\mathrm{k\geq 0}}\sum}\beta_kt^k$ is a formal series in $t\in\RR$, with nonnegative integer coefficients.

Finally, if $x\in\RR$, we set $x^{\pm}=\max\{\pm x,0\}$. Then for $u\in W^{1,p}_{0}(\Omega)$, we set $u^{\pm}(\cdot)=u(\cdot)^{\pm}$. We know that
$$u^{\pm}\in W^{1,p}_{0}(\Omega),\ u=u^+-u^-,\ |u|=u^++u^-.$$

Also, if $h:\Omega\times\RR\rightarrow\RR$ is a measurable function (for example, a Carath\'eodory function), then we define
$$N_h(u)(\cdot)=h(\cdot,u(\cdot))\ \mbox{for all}\ u\in W^{1,p}_{0}(\Omega)$$
(the Nemytskii map corresponding to $h$). Note that $z\mapsto N_h(u)(z)$ is measurable. We denote by $|\cdot|_N$ the Lebesgue measure on $\RR^N$.

\section{The nonlinear equation $(1<p<\infty)$}

In this section we deal with the general equation (\ref{eq1}) and prove two multiplicity theorems producing three nontrivial solutions, all with sign information. The two multiplicity theorems differ in the geometry near the origin. In the first one, the reaction is $(p-1)$-sublinear near zero, while in the second, it is $(p-1)$-superlinear (we have the presence of a concave term).

For the first multiplicity theorem, we start with the following hypotheses on the reaction $f(z,x)$. Using them, we will generate two nontrivial constant sign solutions:

\smallskip
$H_1:$ $f:\Omega\times\RR\rightarrow\RR$ is a measurable function such that for almost all $z\in \Omega$, $f(z,0)=0$, $f(z,\cdot)$ is locally $\alpha$-H\"{o}lder continuous with $\alpha\in\left(0,1\right]$ and local H\"{o}lder constant $k\in L^{\infty}(\Omega)_+$ and
\begin{itemize}
	\item[(i)] for every $\rho>0$, there exists $a_{\rho}\in L^{\infty}(\Omega)$ such that
	$$|f(z,x)|\leq a_{\rho}(z)\ \mbox{for almost all}\ z\in\Omega,\ \mbox{and all}\ |x|\leq\rho;$$
	\item[(ii)] $\limsup\limits_{x\rightarrow\pm \infty}\frac{f(z,x)}{|x|^{p-2}x}\leq\xi<\hat{\lambda}_1$ uniformly for almost all $z\in\Omega$;
	\item[(iii)] there exists a function $\eta\in L^{\infty}(\Omega)$ such that
	\begin{eqnarray*}
		&&\eta(z)\geq\hat{\lambda}_1\ \mbox{for almost all}\ z\in\Omega,\ \mbox{the inequality is strict on a set of positive measure,}\\
		&&\liminf\limits_{x\rightarrow 0}\frac{f(z,x)}{|x|^{p-2}x}\geq\eta(z)\ \mbox{uniformly for almost all}\ z\in\Omega;
	\end{eqnarray*}
\item[(iv)] there exists $M_0>0$ such that for almost all $z\in\Omega$,
\begin{eqnarray*}
		&& x\mapsto\frac{f(z,x)}{x^{p-1}}\ \mbox{is nondecreasing on $[M_0,+\infty)$};\\
		&&x\mapsto\frac{f(z,x)}{|x|^{p-2}x}\ \mbox{is nonincreasing on $(-\infty, -M_0]$}.
	\end{eqnarray*}
\end{itemize}
\begin{remark}
	We stress that the above conditions do not impose any global growth condition from below on the reaction $f(z,\cdot)$.
\end{remark}

Hypothesis $H_1(ii)$ implies that we can find $\xi_1\in(\xi,\hat{\lambda}_1)$ and $M\geq M_0$ such that
\begin{equation}\label{eq8}
	f(z,x)x\leq\xi_1|x|^p\ \mbox{for almost all}\ z\in\Omega,\ \mbox{and all}\ |x|\geq M.
\end{equation}

Also, let $\beta\in L^{\infty}(\Omega)_+$ such that
\begin{equation}\label{eq9}
	\beta(z)\geq a_M(z)+1\ \mbox{for almost all}\ z\in\Omega\ (\mbox{see hypothesis}\ H_1(i)).
\end{equation}

Let $\{t_n\}_{n\geq 1}\subseteq\left[1,+\infty\right)$ and assume that $t_n\rightarrow+\infty$. We define
$$h_n(z)=\left\{\begin{array}{ll}
	\hat{\lambda}_1(t_n\hat{u}_1(z))^{p-1}&\mbox{if}\ z\in\{t_n\hat{u}_1> M\}\\
	t_n^{p-1}\beta(z)&\mbox{if}\ z\in\{t_n\hat{u}_1\leq M\}.
\end{array}\right.$$

Evidently, $h_n\in L^{\infty}(\Omega)$ for all $n\geq 1$. Recall that $\hat{u}_1\in {\rm int}\,C_+$. Hence $\{t_n\hat{u}_1\leq M\}\downarrow\emptyset$ as $n\rightarrow\infty$. So, for every $r\in\left[1,\infty\right)$ we have
\begin{equation}\label{eq10}
	\left\|\frac{h_n}{t_n^{p-1}}-\hat{\lambda}_1\hat{u}_1^{p-1}\right\|_r\rightarrow 0\ \mbox{as}\ n\rightarrow\infty.
\end{equation}

On the other hand, by Gasinski \& Papageorgiou \cite[p. 477]{11}, we know that
\begin{equation}\label{eq11}
\left\|\frac{h_n}{t_n^{p-1}}-\hat{\lambda}_1\hat{u}_1^{p-1}\right\|_r\rightarrow
\left\|\frac{h_n}{t_n^{p-1}}-\hat{\lambda}_1\hat{u}_1^{p-1}\right\|_{\infty}\ \mbox{as}\ r\rightarrow\infty,\ \mbox{for every}\ n\geq 1.
\end{equation}

Then from (\ref{eq11}) we see that given $\epsilon>0$, we can find $r_0=r_0(\epsilon)\in\NN$ such that
\begin{equation}\label{eq12}
	\left\|\frac{h_n}{t_n^{p-1}}-\hat{\lambda}_1\hat{u}_1^{p-1}\right\|_{\infty}\leq
\left\|\frac{h_n}{t_n^{p-1}}-\hat{\lambda}_1\hat{u}_1^{p-1}\right\|_r+\frac{\epsilon}{2}\ \mbox{for all}\ r\geq r_0.
\end{equation}

Fix $r\geq r_0$. From (\ref{eq10}) we see that we can find $n_0=n_0(\epsilon)\in\NN$ such that
\begin{equation}\label{eq13}
	\left\|\frac{h_n}{t_n^{p-1}}-\hat{\lambda}_1\hat{u}_1^{p-1}\right\|_{r}\leq\frac{\epsilon}{2}\ \mbox{for all}\ n\geq n_0.
\end{equation}

For the fixed $r\geq r_0$, using (\ref{eq13}) in (\ref{eq12}), we obtain
\begin{eqnarray}\label{eq14}
	&&\left\|\frac{h_n}{t_n^{p-1}}-\hat{\lambda}_1\hat{u}_1^{p-1}\right\|_{\infty}\leq\epsilon\ \mbox{for all}\ n\geq n_0,\nonumber\\
	&\Rightarrow&\frac{h_n}{t_n^{p-1}}\rightarrow\hat{\lambda}_1\hat{u}_1^{p-1}\ \mbox{in}\ L^{\infty}(\Omega)\ \mbox{as}\ n\rightarrow\infty.
\end{eqnarray}

Then for every $n\geq 1$, we consider the following auxiliary Dirichlet problem
$$-\Delta_pu_n(z)=h_n(z)\ \mbox{in}\ \Omega,\ u_n|_{\partial\Omega}=0.$$

This problem has a unique solution $u_n\in W^{1,p}_{0}(\Omega),\ u_n\geq 0$. The nonlinear regularity theory and the nonlinear maximum principle (see \cite[pp. 737-738]{10}), imply that $u_n\in {\rm int}\,C_+$ for all $n\geq 1$. Let $v_n=\frac{u_n}{t_n}$ for all $n\geq 1$. We have
$$-\Delta_pv_n(z)=\frac{h_n(z)}{t_n^{p-1}}\ \mbox{in}\ \Omega,\ v_n|_{\partial\Omega}=0.$$

From Gasinski \& Papageorgiou \cite[p. 738]{10}, we know that we can find $\theta\in(0,1)$ and $M_1>0$ such that
\begin{equation}\label{eq15}
	v_n\in C^{1,\theta}_{0}(\overline{\Omega})\ \mbox{and}\ ||v_n||_{C^{1,\theta}_{0}(\overline{\Omega})}\leq M_1\ \mbox{for all}\ n\geq 1.
\end{equation}

Exploiting the compact embedding of $C^{1,\theta}_{0}(\overline{\Omega})$ into $C^1_0(\overline{\Omega})$ and using (\ref{eq14}),  we can infer from (\ref{eq15}) that
\begin{equation}\label{eq16}
	v_n\rightarrow \hat{u}_1\ \mbox{in}\ C^1_0(\overline{\Omega})\ \mbox{as}\ n\rightarrow\infty\,.
\end{equation}

Hence by (\ref{eq16}), we can find $n_1\in\NN$ such that
\begin{equation}\label{eq17}
	\xi_1v_n(z)^{p-1}\leq\hat{\lambda}_1\hat{u}_1(z)^{p-1}\ \mbox{for all}\ z\in\overline{\Omega},\  n\geq n_1
\end{equation}
and if $t_n\hat{u}_1(z)>M$, then $t_nv_n(z)>M$ for all $n\geq n_1$.

Also, by (\ref{eq16}) and our hypothesis on $f(z,\cdot)$, we can find $n_2\in\NN$ such that
\begin{eqnarray}\label{eq18}
	&&|f(z,\hat{u}_1(z))-f(z,v_n(z))|\leq||k||_{\infty}||\hat{u}_1-v_n||^{\alpha}_{\infty}\leq 1\\
	&&\hspace{4.3cm}\mbox{for all}\ z\in\Omega,\ n\geq n_2.\nonumber
\end{eqnarray}

Let $n_0=\max\{n_1,n_2\}$. Then for $n\geq n_0$ we have:
\begin{eqnarray*}
	&&\mbox{If}\ z\in\{t_n\hat{u}_1> M\},\ \mbox{then}\\
	 &&-\Delta_p(t_nv_n)(z)=\hat{\lambda}_1((t_nu_n(z))^{p-1}\geq\xi_1(t_nv_n(z))^{p-1}\geq f(z,t_nv_n(z))\ \mbox{(see (\ref{eq8}) and (\ref{eq17}))},\\
&&\Longrightarrow -\Delta_pv_n(z)\geq f(z,v_n(z))\ \mbox{(see hypotheses $H_1(iv)$ and recall $t_n\geq 1$)}.\\
	&&\mbox{If}\ z\in\{t_n\hat{u}_1\leq M\},\ \mbox{then}\\
	&&-\Delta_pv_n(z)=\frac{h(z)}{t_n^{p-1}}=\beta(z)\geq f(z,\hat{u}_1(z))+1\geq f(z,v_n(z))\ (\mbox{see (\ref{eq9}) and (\ref{eq18})}).
\end{eqnarray*}

So, fixing $n\geq n_0$ and setting $\bar{u}=v_n\in {\rm int}\,C_+$, we have
\begin{equation}\label{eq19}
	-\Delta_p\bar{u}(z)\geq f(z,\bar{u}(z))\ \mbox{for almost all}\ z\in\Omega.
\end{equation}

In a similar fashion, we produce $\bar{v}\in-{\rm int}\,C_+$ such that
\begin{equation}\label{eq20}
	-\Delta_p\bar{v}(z)\leq f(z,\bar{v}(z))\ \mbox{for almost all}\ z\in\Omega.
\end{equation}

Now, we are ready to produce nontrivial constant sign solutions for problem (\ref{eq1}).
\begin{prop}\label{prop6}
	Assume that hypotheses $H_1$ hold. Then problem (\ref{eq1}) admits at least two constant sign solutions
	$$u_0\in[0,\bar{u}]\cap {\rm int}\,C_+\ \mbox{and}\ v_0\in[\bar{v},0]\cap(-{\rm int}\,C_+)$$
	(here $[0,\bar{u}]=\{u\in W^{1,p}_{0}(\Omega):0\leq u(z)\leq\bar{u}(z)\ \mbox{for almost all}\ z\in\Omega\}$ and $[\bar{v},0]=\{u\in W^{1,p}_{0}(\Omega):\bar{v}(z)\leq u(z)\leq 0\ \mbox{for almost all}\ z\in\Omega\}$).
\end{prop}
\begin{proof}
	First, we produce the positive solution. To this end, we consider the following truncation of $f(z,\cdot)$:
	\begin{eqnarray}\label{eq21}
		\hat{f}_+(z,x)=\left\{\begin{array}{ll}
			0&\mbox{if}\ x<0\\
			f(z,x)&\mbox{if}\ 0\leq x\leq \bar{u}(z)\\
			f(z,\bar{u}(z))&\mbox{if}\ \bar{u}(z)<x.
		\end{array}\right.
	\end{eqnarray}
	
	This is a Carath\'eodory function. We set $\hat{F}_+(z,x)=\int^{x}_{0}\hat{f}_+(z,s)ds$ and consider the $C^1$-functional $\hat{\varphi}_+:W^{1,p}_{0}(\Omega)\rightarrow\RR$ defined by
	$$\hat{\varphi}_+(u)=\frac{1}{p}||Du||^p_p-\int_{\Omega}\hat{F}_+(z,u(z))dz\ \mbox{for all}\ u\in W^{1,p}_{0}(\Omega).$$
	
	From (\ref{eq21}) it is clear that $\hat{\varphi}_+$ is coercive. Also, using the Sobolev embedding theorem, we can easily check that $\hat{\varphi}_+$ is sequentially weakly lower semicontinuous. So, by the Weierstrass theorem, we can find $u_0\in W^{1,p}_{0}(\Omega)$ such that
	\begin{equation}\label{eq22}
		\hat{\varphi}_+(u_0)=\inf\{\hat{\varphi}_+(u):u\in W^{1,p}_{0}(\Omega)\}=\hat{m}_+.
	\end{equation}
	
	By virtue of hypothesis $H_1(iii)$, given $\epsilon>0$, we can find $\delta=\delta(\epsilon)>0$ such that
	\begin{equation}\label{eq23}
		F(z,x)\geq\frac{1}{p}(\eta(z)-\epsilon)x^p\ \mbox{for almost all}\ z\in\Omega,\ \mbox{and all}\ x\in[0,\delta].
	\end{equation}
	
	Here, $F(z,x)=\int^{x}_{0}f(z,s)ds$. Since $\hat{u}_1\in {\rm int}\,C_+$, we can find small enough $t\in(0,1)$ so that $t\hat{u}_1(z)\in[0,\delta]$ for all $z\in\overline{\Omega}$. We have
	\begin{eqnarray}\label{eq24}
		 \hat{\varphi}_+(t\hat{u}_1)&\leq&\frac{t^p}{p}\hat{\lambda}_1||\hat{u}_1||^p_p-\frac{t^p}{p}\int_{\Omega}\eta(z)\hat{u}_1^pdz+\frac{t^p}{p}\epsilon\ (\mbox{recall}\ ||\hat{u}_1||_p=1\ \mbox{and see (\ref{eq23})})\nonumber\\
		 &=&\frac{t^p}{p}\left[\int_{\Omega}(\hat{\lambda}_1-\eta(z))\hat{u}_1^pdz+\epsilon\right].
	\end{eqnarray}
	
	Note that
	$$\xi_0=\int_{\Omega}(\eta(z)-\hat{\lambda}_1)\hat{u}_1^pdz>0.$$
	
	So, if we choose $\epsilon\in(0,\xi_0)$, then from (\ref{eq24}) we see that
	\begin{eqnarray*}
		&&\hat{\varphi}_+(t\hat{u}_1)<0,\\
		&\Rightarrow&\hat{\varphi}_+(u_0)<0=\hat{\varphi}_+(0)\ \mbox{(see (\ref{eq22}), hence)}\ u_0\neq 0.
	\end{eqnarray*}
	
	From (\ref{eq22}) we have
	\begin{eqnarray}\label{eq25}
		&&\hat{\varphi}'_+(u_0)=0,\nonumber\\
		&\Rightarrow&A(u_0)=N_{\hat{f}_+}(u_0).
	\end{eqnarray}
	
	On (\ref{eq25})  we first act with $-u_0^-\in W^{1,p}_{0}(\Omega)$. We obtain
	\begin{eqnarray*}
		&&||Du_0^-||^p_p=0\ (\mbox{see (\ref{eq21})}),\\
		&\Rightarrow&u_0\geq 0,\ u_0\neq 0.
	\end{eqnarray*}
	
	Then  we act on (\ref{eq25}) with $(u_0-\bar{u})^+\in W^{1,p}_{0}(\Omega)$. We have
	\begin{eqnarray*}
		&&\left\langle A(u_0),(u_0-\bar{u})^+\right\rangle=\int_{\Omega}\hat{f}_+(z,u_0)(u_0-\bar{u})^+dz\\
		&&\hspace{3.2cm}=\int_{\Omega}f(z,\bar{u})(u_0-\bar{u})^+dz\ (\mbox{see (\ref{eq21})})\\
		&&\hspace{3.2cm}\leq\left\langle A(\bar{u}),(u_0-\bar{u})^+\right\rangle\ (\mbox{see (\ref{eq19})}),\\
		 &\Rightarrow&\int_{\{u_0>\bar{u}\}}(|Du_0|^{p-2}Du_0-|D\bar{u}|^{p-2}D\bar{u},Du_0-D\bar{u})_{\RR^N}dz\leq 0,\\
		&\Rightarrow&|\{u_0>\bar{u}\}|_N=0,\ \mbox{hence}\ u_0\leq\bar{u}.
	\end{eqnarray*}
	
	So, we have proved that
	\begin{equation}\label{eq26}
		u_0\in [0,\bar{u}],\ u_0\neq 0.
	\end{equation}
	
	Then (\ref{eq25}) becomes
	\begin{eqnarray}\label{eq27}
		&&A(u_0)=N_f(u_0)\ (\mbox{see (\ref{eq21}) and (\ref{eq26})}),\nonumber\\
		&\Rightarrow&-\Delta_pu_0(z)=f(z,u_0(z))\ \mbox{for almost all}\ z\in\Omega,\ u_0|_{\partial\Omega}=0.
	\end{eqnarray}
	
	The nonlinear regularity theory (see \cite[pp. 737-738]{10}) implies that $u_0\in C_+\backslash\{0\}$. Let $\rho=||u_0||_{\infty}$. Hypotheses $H_1(i), (iii)$ imply that we can find $\hat{\xi}_{\rho}>0$ such that
	\begin{equation}\label{eq28}
		f(z,x)+\hat{\xi}_{\rho}x^{p-1}\geq 0\ \mbox{for almost all}\ z\in \Omega,\ \mbox{and all}\ x\in[0,\rho].
	\end{equation}
	
	Then from (\ref{eq27}) and (\ref{eq28}), we have
	\begin{eqnarray*}
		&&\Delta_pu_0(z)\leq\hat{\xi}_{\rho}u_0(z)^{p-1}\ \mbox{for almost all}\ z\in\Omega,\\
		&\Rightarrow&u_0\in {\rm int}\,C_+\ (\mbox{by the nonlinear maximum principle, see \cite[p. 738]{9}}).
	\end{eqnarray*}
	
	Similarly, for the negative solution, we introduce the truncation
	\begin{eqnarray}\label{eq29}
		\hat{f}_-(z,x)=\left\{\begin{array}{ll}
			f(z,\bar{v}(z))&\mbox{if}\ x<\bar{v}(z)\\
			f(z,x)&\mbox{if}\ \bar{v}(z)\leq x\leq 0\\
			0&\mbox{if}\ 0<x.
		\end{array}\right.
	\end{eqnarray}
	
	This is a Carath\'eodory function. We set $\hat{F}_-(z,x)=\int^{x}_{0}\hat{f}_-(z,s)ds$ and consider the $C^1$-functional $\hat{\varphi}_-:W^{1,p}_{0}(\Omega)\rightarrow\RR$ defined by
	$$\hat{\varphi}_-(u)=\frac{1}{p}||Du||^p_p-\int_{\Omega}\hat{F}_-(z,u(z))dz\ \mbox{for all}\ u\in W^{1,p}_{0}(\Omega).$$
	
	Working with $\hat{\varphi}_-$ as above, via the direct method and using \eqref{eq20}, we produce a negative solution
	$$v_0\in[\bar{v},0]\cap(-{\rm int}\,C_+).$$ The proof is now complete.
\end{proof}

In fact, we can produce extremal constant sign solutions, that is, a smallest positive and a biggest negative solutions. These extremal solutions will be helpful in obtaining nodal ones.
\begin{prop}\label{prop7}
	Assume that hypotheses $H_1$ hold. Then  problem (\ref{eq1}) admits a smallest positive solution $u_*\in {\rm int}\, C_+$ and a biggest negative solution $v_*\in-{\rm int}\,C_+.$
\end{prop}
\begin{proof}
	First we produce the smallest positive solution. Let $S_+$ be the set of positive solutions of problem (\ref{eq1}). From Proposition \ref{prop6} and its proof, we know that $S_+\cap[0,\bar{u}]\neq\varnothing$ and $S_+\subseteq {\rm int}\,C_+$. By Hu \& Papageorgiou \cite[p. 178]{12}, we know that we can find $\{u_n\}_{n\geq 1}\subseteq S_+\cap[0,\bar{u}]$ such that
	$$\inf S_+=\inf\limits_{n\geq 1} u_n.$$
	
	We have
	\begin{eqnarray}\label{eq30}
		&&A(u_n)=N_f(u_n), u_n\leq\bar{u}\ \mbox{for all}\ n\geq 1,\\
		&\Rightarrow&\{u_n\}_{n\geq 1}\subseteq W^{1,p}_{0}(\Omega)\ \mbox{is bounded}\ (\mbox{see hypothesis}\ H_1(i)).\nonumber
	\end{eqnarray}
	
	So, we may assume that
	\begin{equation}\label{eq31}
		u_n\stackrel{w}{\rightarrow}u_*\ \mbox{in}\ W^{1,p}_{0}(\Omega)\ \mbox{and}\ u_n\rightarrow u_*\ \mbox{in}\ L^p(\Omega).
	\end{equation}
	
	On (\ref{eq30}) we act with $u_n-u_*\in W^{1,p}_{0}(\Omega)$, pass to the limit as $n\rightarrow\infty$ and use (\ref{eq25}). Then
	\begin{eqnarray}\label{eq32}
		&&\lim\limits_{n\rightarrow\infty}\left\langle A(u_n),u_n-u_*\right\rangle=0,\nonumber\\
		&\Rightarrow&u_n\rightarrow u_*\ \mbox{in}\ W^{1,p}_{0}(\Omega)\ (\mbox{see Proposition \ref{prop5}}).
	\end{eqnarray}
	
	So, if in (\ref{eq30}) we pass to the limit as $n\rightarrow\infty$ and use (\ref{eq32}), then
	\begin{eqnarray*}
		&&A(u_*)=N_f(u_*),\\
		&\Rightarrow&u_*\ \mbox{is a nonnegative solution of (\ref{eq1}) and}\ u_*\in C_+\\
		&&\hspace{2cm}(\mbox{nonlinear regularity theory, see \cite[p. 738]{9}}).
	\end{eqnarray*}
	
	We need to show that $u_*\neq 0$. By virtue of hypotheses $H_1(i),(iii)$, given $\epsilon>0$, we can find $c_1=c_1(\epsilon)>0$ such that
	\begin{equation}\label{eq33}
		f(z,x)x\geq(\eta(z)-\epsilon)|x|^p-c_1|x|^r\ \mbox{for almost all}\ z\in\Omega,\ \mbox{and all}\ |x|\leq\rho,
	\end{equation}
	with $r>p$ and $\rho=\max\{||\bar{u}||_{\infty},||\bar{v}||_{\infty}\}$. We introduce the following Carath\'eodory functions
	\begin{eqnarray}
		&&g_+(z,x)=\left\{\begin{array}{ll}
			0&\mbox{if}\ x<0\\
			(\eta(z)-\epsilon)x^{p-1}-c_1x^{r-1}&\mbox{if}\ 0\leq x\leq \bar{u}(z)\\
			(\eta(z)-\epsilon)\bar{u}(z)^{p-1}-c_1\bar{u}(z)^{r-1}&\mbox{if}\ \bar{u}(z)<x
		\end{array}\right.\label{eq34}\\
		\mbox{and}&&g_-(z,x)=\left\{\begin{array}{ll}
			 (\eta(z)-\epsilon)|\bar{v}(z)|^{p-2}\bar{v}(z)-c_1|\bar{v}(z)|^{r-2}\bar{v}(z)&\mbox{if}\ x<\bar{v}(z)\\
			(\eta(z)-\epsilon)|x|^{p-2}x-c_1|x|^{r-2}x&\mbox{if}\ \bar{v}(z)\leq x\leq 0\\
			0&\mbox{if}\ 0<x.
		\end{array}\right.\label{eq35}
	\end{eqnarray}
	
	We consider the following auxiliary Dirichlet problems:
	\begin{eqnarray}
		&&-\Delta_pu(z)=g_+(z,u(z))\ \mbox{in}\ \Omega,\quad u|_{\partial\Omega}=0,\label{eq36}\\
		&&-\Delta_pv(z)=g_-(z,v(z))\ \mbox{in}\ \Omega,\quad v|_{\partial\Omega}=0.\label{eq37}
	\end{eqnarray}
	
	\begin{claim}\label{cl1}
		Problem (\ref{eq36}) (resp. Problem (\ref{eq37})) for $\epsilon>0$ small admits a unique positive solution $\tilde{u}\in {\rm int}\,C_+$ (resp. a unique negative solution $\tilde{v}\in-{\rm int}\, C_+$).
	\end{claim}
	
	First, we deal with problem (\ref{eq36}). So, let $\psi_+:W^{1,p}_{0}(\Omega)\rightarrow\RR$ be the $C^1$-functional defined by
	$$\psi_+(u)=\frac{1}{p}||Du||^p_p-\int_{\Omega}G_+(z,u(z))dz\ \mbox{for all}\ u\in W^{1,p}_{0}(\Omega),$$
	where $G_+(z,x)=\int^{x}_{0}g_+(z,s)ds$. From (\ref{eq34}) it is clear that $\psi_+$ is coercive. Also, it is sequentially weakly lower semicontinuous. So, we can find $\tilde{u}\in W^{1,p}_{0}(\Omega)$ such that
	\begin{equation}\label{eq38}
		\psi_+(\tilde{u})=\inf\{\psi_+(u):u\in W^{1,p}_{0}(\Omega)\}.
	\end{equation}
	
	Let $t\in(0,1)$ be small such that $t\hat{u}_1\leq\bar{u}$ (recall that $\bar{u}\in {\rm int}\,C_+$ and use Lemma 3.3 of Filippakis, Kristaly \& Papageorgiou \cite{8}). We have
	\begin{eqnarray*}
		 \psi_+(t\hat{u}_1)&\leq&\frac{t^p}{p}\hat{\lambda}_1+
\frac{c_1}{r}t^r||\hat{u}_1||^r_r-\frac{t^p}{p}\int_{\Omega}(\eta(z)-\epsilon)\hat{u}_1^pdz\ (\mbox{see (\ref{eq34})})\\
		 &=&\frac{t^p}{p}\left[\int_{\Omega}(\hat{\lambda}_1-(\eta(z)-\epsilon))\hat{u}_1^pdz\right]+\frac{c_1}{r}t^r||\hat{u}_1||^r_r.
	\end{eqnarray*}
	
	Note that $\beta=\int_{\Omega}(\eta(z)-\hat{\lambda}_1)\hat{u}_1^pdz>0$. So, choosing $\epsilon\in(0,\beta)$, we obtain
	$$\psi_+(t\hat{u}_1)\leq-\frac{t^p}{p}c_2+\frac{t^r}{r}c_1||\hat{u}_1||^r_r\ \mbox{for some}\ c_2>0.$$
	
	Since $r>p$, by choosing $t\in(0,1)$ even smaller if necessary, we obtain
	\begin{eqnarray*}
		&&\psi_+(t\hat{u}_1)<0,\\
		&\Rightarrow&\psi_+(\tilde{u})<0=\psi_+(0)\ \mbox{(see (\ref{eq38})), hence}\ \tilde{u}\neq 0.
	\end{eqnarray*}
	
	From (\ref{eq38}), we have
	\begin{eqnarray}\label{eq39}
		&&\psi_+'(\tilde{u})=0,\nonumber\\
		&\Rightarrow&A(\tilde{u})=N_{g_+}(\tilde{u}).
	\end{eqnarray}
	
	On (\ref{eq39}) first we act with $-\tilde{u}^-\in W^{1,p}_{0}(\Omega)$. Then
	\begin{eqnarray*}
		&&||D\tilde{u}^-||^p_p=0\ \mbox{(see (\ref{eq34}))},\\
		&\Rightarrow&\tilde{u}\geq 0,\ \tilde{u}\neq 0.
	\end{eqnarray*}
	
	Also, we act  on (\ref{eq39}) with $(\tilde{u}-\bar{u})^+\in W^{1,p}_{0}(\Omega)$. Then
	\begin{eqnarray*}
		&&\left\langle A(\tilde{u}),(\tilde{u}-\bar{u})^+\right\rangle=\int_{\Omega}\left[(\eta(z)-\epsilon)\bar{u}^{p-1}-c_1\bar{u}^{r-1}\right](\tilde{u}-\bar{u})^+dz\ (\mbox{see (\ref{eq34})})\\
		&&\hspace{3cm}\leq\int_{\Omega}f(z,\bar{u})(\tilde{u}-\bar{u})^+dz\ (\mbox{see (\ref{eq33})})\\
		&&\hspace{3cm}\leq\left\langle A(\bar{u}),(\tilde{u}-\bar{u})^+\right\rangle\ (\mbox{see (\ref{eq19})}),\\
		 &\Rightarrow&\int_{\{\tilde{u}>\bar{u}\}}(|D\tilde{u}|^{p-2}D\tilde{u}-|D\bar{u}|^{p-2}D\bar{u},D\tilde{u}-D\bar{u})_{\RR^N}dz\leq0,\\
		&\Rightarrow&|\{\tilde{u}>\bar{u}\}|_N=0,\\
		&\Rightarrow&\tilde{u}\leq\bar{u}.
	\end{eqnarray*}
	
	So, we have proved that
	\begin{equation}\label{eq40}
		\tilde{u}\in[0,\bar{u}],\ \tilde{u}\neq 0.
	\end{equation}
	
	From (\ref{eq34}) and (\ref{eq40}), equation (\ref{eq39}) becomes
	\begin{eqnarray*}
		&&A(\tilde{u})=(\eta(\cdot)-\epsilon)\tilde{u}^{p-1}-c_1\tilde{u}^{r-1},\\
		&\Rightarrow&\tilde{u}\ \mbox{is a positive solution of (\ref{eq36})}.
	\end{eqnarray*}
	
	The nonlinear regularity theory and the nonlinear maximum principle (see \cite[pp. 737-738]{10}) imply
	$$\tilde{u}\in {\rm int}\,C_+.$$
	
	Now we show that $\tilde{u}$ is the unique positive solution of (\ref{eq36}). To this end, let $\tilde{y}$ be another positive solution of (\ref{eq36}). As we did for $\tilde{u}$, we can show that $\tilde{y}\in[0,\bar{u}]\cap {\rm int}\,C_+$. Note that we can find $c_3>0$ such that for almost all $z\in\Omega$ the function $x\mapsto (\eta(z)+c_3-\epsilon)x^{p-1}-c_1x^{r-1}$ is nondecreasing on $[0,\rho]$ (recall $\rho=\max\{||\bar{u}||_{\infty},||\bar{v}||_{\infty}\}$). Let $t>0$ be the biggest positive real such that
	$$t\tilde{y}\leq\tilde{u}\ \mbox{(see Filippakis, Kristaly \& Papageorgiou \cite[Lemma 3.3]{8})}.$$
	
	Suppose $t\in(0,1)$. We have
	\begin{eqnarray*}
		&&-\Delta_p(t\tilde{y})+c_3(t\tilde{y})^{p-1}\\
		&=&t^{p-1}\left[-\Delta_p\tilde{y}+c_3\tilde{y}^{p-1}\right]\\
		&=&t^{p-1}\left[(\eta(z)+c_3-\epsilon)\tilde{y}^{p-1}-c_1\tilde{y}^{r-1}\right]\\
		&<&(\eta(z)-\epsilon)(t\tilde{y})^{p-1}-c_1(t\tilde{y})^{r-1}+c_3(t\tilde{y})^{p-1}\ (\mbox{since}\ r>p,\ t\in(0,1))\\
		&\leq&(\eta(z)-\epsilon)\tilde{u}^{p-1}-c_1\tilde{u}^{r-1}+c_3\tilde{u}^{p-1}\ (\mbox{since}\ t\tilde{y}\leq\tilde{u}\ \mbox{and the choice of $c_3$})\\
		&=&-\Delta_p\tilde{u}+c_3\tilde{u}^{p-1}\ (\mbox{since $\tilde{u}\in {\rm int}\,C_+$ is a solution of (\ref{eq36})}),\\
		\Rightarrow&&\tilde{u}-t\tilde{y}\in {\rm int}\,C_+\ (\mbox{see Arcoya \& Ruiz \cite[Proposition 2.6]{5}}).
	\end{eqnarray*}
	
	This contradicts the maximality of $t>0$. Therefore $t\geq 1$ and so
	$$\tilde{y}\leq\tilde{u}.$$
	
	If in the above argument we interchange the roles of $\tilde{y}$ and $\tilde{u}$, we also have
	\begin{eqnarray*}
		&&\tilde{u}\leq\tilde{y},\\
		&\Rightarrow&\tilde{u}=\tilde{y}.
	\end{eqnarray*}
	
	This proves the uniqueness of the solution $\tilde{u}\in {\rm int}\,C_+$ of problem (\ref{eq36}).
	
	Similarly, using the $C^1$-functional $\psi_-:W^{1,p}_{0}(\Omega)\rightarrow\RR$ defined by
	$$\psi_-(u)=\frac{1}{p}||Du||^p_p-\int_{\Omega}G_-(z,u(z))dz\ \mbox{for all}\ u\in W^{1,p}_{0}(\Omega),$$
	where $G_-(z,x)=\int^{x}_{0}g_-(z,s)ds$ and reasoning as above, we show that problem (\ref{eq37}) has a unique solution $\tilde{v}\in-{\rm int}\,C_+$. This proves Claim \ref{cl1}.
	\begin{claim}\label{cl2}
		$\tilde{u}\leq u$ for all $u\in S_+\cap[0,\bar{u}]$.
	\end{claim}
	
	Let $u\in S_+\cap[0,\bar{u}]\subseteq[0,\bar{u}]\cap {\rm int}\,C_+$ and consider the Carath\'eodory function
	\begin{eqnarray}\label{eq41}
		k_+(z,x)=\left\{\begin{array}{ll}
			0&\mbox{if}\ x<0\\
			(\eta(z)-\epsilon)x^{p-1}-c_1x^{r-1}&\mbox{if}\ 0\leq x\leq u(z)\\
			(\eta(z)-\epsilon)u(z)^{p-1}-c_1u(z)^{r-1}&\mbox{if}\ u(z)<x.
		\end{array}\right.
	\end{eqnarray}
	
	Let $K_+(z,x)=\int^{x}_{0}k_+(z,s)ds$ and consider the $C^1$-functional $\sigma_+:W^{1,p}_{0}(\Omega)\rightarrow\RR$ defined by
	$$\sigma_+(u)=\frac{1}{p}||Du||^p_p-\int_{\Omega}K_+(z,u(z))dz\ \mbox{for all}\ u\in W^{1,p}_{0}(\Omega).$$
	
	From (\ref{eq41}) we see that $\sigma_+$ is coercive. Also, it is sequentially weakly continuous. So, we can find $\tilde{u}_*\in W^{1,p}_{0}(\Omega)$ such that
	\begin{equation}\label{eq42}
		\sigma_+(\tilde{u}_*)=\inf\{\sigma_+(u):u\in W^{1,p}_{0}(\Omega)\}.
	\end{equation}
	
	As in the proof of Claim \ref{cl1}, we can show that for $t\in(0,1)$ small (at least such that $t\hat{u}_1\leq u\in {\rm int}\,C_+$), we have
	\begin{eqnarray*}
		&&\sigma_+(t\hat{u}_1)<0=\sigma_+(0),\\
		&\Rightarrow&\sigma_+(\tilde{u}_*)<0=\sigma_+(0)\ \mbox{(see (\ref{eq42})), hence}\ \tilde{u}_*\neq 0.
	\end{eqnarray*}
	
	As before, we can check that
	\begin{eqnarray*}
		&&K_{\sigma_+}\subseteq[0,u]\subseteq[0,\bar{u}],\\
		&\Rightarrow&\tilde{u}_*\in[0,u]\backslash\{0\}\ (\mbox{see (\ref{eq42})}),\\
		&\Rightarrow&\tilde{u}_*=\tilde{u}\in {\rm int}\,C_+\ (\mbox{see Claim \ref{cl1} and (\ref{eq41})}),\\
		&\Rightarrow&\tilde{u}\leq u\ \mbox{for all}\ u\in C_+\cap[0,\bar{u}].
	\end{eqnarray*}
	
	This proves Claim \ref{cl2}.
	
	Because of Claim \ref{cl2}, we have
	\begin{eqnarray*}
		&&\tilde{u}\leq u_n\ \mbox{for all}\ n\geq 1,\\
		&\Rightarrow&\tilde{u}\leq u_*\ (\mbox{see (\ref{eq32})})\\
		&\Rightarrow&u_*\neq 0.
	\end{eqnarray*}
	
	Hence we have
	$$u_*\in S_+\ \mbox{and}\ u_*=\inf S_+.$$
	
	Similarly, if $S_-$ is the set of negative solutions of (\ref{eq1}), we produce $v_*\in-{\rm int}\,C_+$ the biggest element of $S_-$. In this case, by Claim \ref{cl2} we have $v\leq\tilde{v}$ for all $v\in S_-\cap[\bar{v},0]$ with $S_-\subseteq-{\rm int}\,C_+$.
\end{proof}

As we have already mentioned, we will use these extremal solutions to produce a nodal solution. To do this, we need to strengthen the condition on $f(z,\cdot)$ near zero. Note that hypothesis $H_1(iii)$ permits that $f(z,\cdot)$ near zero is either $(p-1)$-linear or $(p-1)$-superlinear. We consider both cases and for both we produce nodal solutions.

First, we deal with the $(p-1)$-linear case. We impose the following conditions on the reaction $f(z,x)$.

\smallskip
$H_2:$ $f:\Omega\times\RR\rightarrow\RR$ is a measurable function such that for almost all $z\in \Omega$, $f(z,0)=0$, $f(z,\cdot)$ is locally $\alpha$-H\"{o}lder continuous with $\alpha\in\left(0,1\right]$ and local H\"{o}lder constant $k\in L^{\infty}(\Omega)_+$ and
\begin{itemize}
	\item[(i)] for every $\rho>0$, there exists $a_{\rho}\in L^{\infty}(\Omega)_+$ such that
	$$|f(z,x)|\leq a_{\rho}(z)\ \mbox{for almost all}\ z\in\Omega,\ \mbox{and all}\ |x|\leq\rho;$$
	\item[(ii)] $\limsup\limits_{x\rightarrow\pm\infty}\frac{f(z,x)}{|x|^{p-2}x}\leq\xi<\hat{\lambda}_1$ uniformly for almost all $z\in\Omega$;
	\item[(iii)] there exist $\xi_*\geq\xi_0>\hat{\lambda}_2$ such that
	$$\xi_0\leq\liminf\limits_{x\rightarrow 0}\frac{f(z,x)}{|x|^{p-2}x}\leq\limsup\limits_{x\rightarrow 0}\frac{f(z,x)}{|x|^{p-2}x}\leq\xi_*\ \mbox{uniformly for almost all}\ z\in\Omega;$$\\
 \item[(iv)] there exists $M_0>0$ such that for almost all $z\in\Omega$,
\begin{eqnarray*}
		&& x\mapsto\frac{f(z,x)}{x^{p-1}}\ \mbox{is nondecreasing on $[M_0,+\infty)$};\\
		&&x\mapsto\frac{f(z,x)}{|x|^{p-2}x}\ \mbox{is nonincreasing on $(-\infty, -M_0]$}.
	\end{eqnarray*}
\end{itemize}

\begin{prop}\label{prop8}
	If hypotheses $H_2$ hold, then problem (\ref{eq1}) admits a nodal solution $y_0\in[v_*,u_*]\cap C^1_0(\overline{\Omega})$ (here $[v_*,u_*]=\{u\in W^{1,p}_{0}(\Omega):v_*(z)\leq u(z)\leq u_*(z)\ \mbox{for almost all}\ z\in\Omega\}$ with $u_*\in {\rm int}\,C_+$ and $v_*\in-{\rm int}\,C_+$ being the extremal constant sign solutions produced in Proposition \ref{prop7}).
\end{prop}
\begin{proof}
	We consider the following Carath\'eodory function
	\begin{eqnarray}\label{eq43}
		h(z,x)=\left\{\begin{array}{ll}
			f(z,v_*(z))&\mbox{if}\ x<v_*(z)\\
			f(z,x)&\mbox{if}\ v_*(z)\leq x\leq u_*(z)\\
			f(z,u_*(z))&\mbox{if}\ u_*(z)<x.
		\end{array}\right.
	\end{eqnarray}
	
	We set $H(z,x)=\int^{x}_{0}h(z,s)ds$ and consider the $C^1$-functional $\beta:W^{1,p}_{0}(\Omega)\rightarrow\RR$ defined by
	$$\beta(u)=\frac{1}{p}||Du||^p_p-\int_{\Omega}H(z,u(z))dz\ \mbox{for all}\ u\in W^{1,p}_{0}(\Omega).$$
	
	We also consider the positive and negative truncations of $h(z,\cdot)$, namely the Carath\'eodory functions
	$$h_{\pm}(z,x)=h(z,\pm x^{\pm}).$$
	
	We set $H_{\pm}(z,x)=\int^{x}_{0}h_{\pm}(z,s)ds$ and consider the $C^1$-functionals $\beta_{\pm}:W^{1,p}_{0}(\Omega)\rightarrow\RR$ defined by
	$$\beta_{\pm}(u)=\frac{1}{p}||Du||^p_p-\int_{\Omega}H_{\pm}(z,u(z))dz\ \mbox{for all}\ u\in W^{1,p}_{0}(\Omega).$$
	\begin{claim}\label{cl3}
		$K_{\beta}\subseteq[v_*,u_*],\ K_{\beta_+}=\{0,u_*\},\ K_{\beta_-}=\{0,v_*\}$.
	\end{claim}
	
	Let $u\in K_{\beta}$. Then
	\begin{equation}\label{eq44}
		A(u)=N_f(u)
	\end{equation}
	
	On (\ref{eq44}), first we act with $(u-u_*)^+\in W^{1,p}_{0}(\Omega)$. Then
	\begin{eqnarray*}
		&&\left\langle A(u),(u-u_*)^+\right\rangle=\int_{\Omega}f(z,u_*)(u-u_*)^+dz\ (\mbox{see (\ref{eq43})})\\
		&&\hspace{3cm}=\left\langle A(u_*),(u-u_*)^+\right\rangle\ (\mbox{since}\ u_*\in S_+),\\
		&\Rightarrow&\int_{\{u>u_*\}}(|Du|^{p-2}Du-|Du_*|^{p-2}Du_*,Du-Du_*)_{\RR^N}dz=0,\\
		&\Rightarrow&|\{u>u^*\}|_N=0,\ \mbox{hence}\ u\leq u^*.
	\end{eqnarray*}
	
	Similarly, acting on (\ref{eq44}) with $(v_*-u)^+\in W^{1,p}_{0}(\Omega)$, we obtain $v_*\leq u$. So, we have
	\begin{eqnarray*}
		&&u\in[v_*,u_*],\\
		&\Rightarrow&K_{\beta}\subseteq[v_*,u_*].
	\end{eqnarray*}
	
	In a similar fashion, we show that
	$$K_{\beta_+}\subseteq[0,u_*]\ \mbox{and}\ K_{\beta_-}\subseteq[v_*,0].$$
	
	The extremality of the solutions $u_*\in {\rm int}\,C_+$ and $v_*\in-{\rm int}\,C_+$ (see Proposition \ref{prop7}) implies that
	$$K_{\beta_+}=\{0,u_*\}\ \mbox{and}\ K_{\beta_-}=\{0,v_*\}.$$
	
	This proves Claim \ref{cl3}.
	\begin{claim}\label{cl4}
		$u_*\in {\rm int}\,C_+$ and $v_*\in-{\rm int}\,C_+$ are local minimizers of $\beta$.
	\end{claim}
	
	From (\ref{eq43}) it is clear that $\beta_+$ is coercive. Also, it is sequentially weakly lower semicontinuous. So, we can find $\hat{u}_*\in W^{1,p}_{0}(\Omega)$ such that
	\begin{equation}\label{eq45}
		\beta_+(\hat{u}_*)=\inf\{\beta_+(u):u\in W^{1,p}_{0}(\Omega)\}.
	\end{equation}
	
	As before, by virtue of hypothesis $H_2(iii)$, we have
	\begin{eqnarray}\label{eq46}
		&&\beta_+(\hat{u}_*)<0=\beta_+(0),\nonumber\\
		&\Rightarrow&\hat{u}_*\neq 0.
	\end{eqnarray}
	
	From (\ref{eq45}) and Claim \ref{cl1}, we have
	\begin{eqnarray}\label{eq47}
		&&\hat{u}_*\in K_{\beta_+}=\{0,u_*\},\nonumber\\
		&\Rightarrow&\hat{u}_*=u_*\in {\rm int}\,C_+\ (\mbox{see (\ref{eq46})}).
	\end{eqnarray}
	
	Note that $\beta|_{C_+}=\beta_+|_{C_+}$. Then from (\ref{eq47}) we see that
	\begin{eqnarray*}
		&&u_*\ \mbox{is a local}\ C^1_0(\overline{\Omega})-\mbox{minimizer of}\ \beta,\\
		&\Rightarrow&u_*\ \mbox{is a local}\ W^{1,p}_{0}(\Omega)-\mbox{minimizer of}\ \beta\ (\mbox{see Proposition \ref{prop2}}).
	\end{eqnarray*}
	
	Similarly for $v_*\in-{\rm int}\,C_+$ using this time the functional $\beta_-$.
	
	This proves Claim \ref{cl4}.
	
	Because of Claim \ref{cl1}, we may assume that $K_{\beta}$ is finite (otherwise we already have an infinity of nodal solutions, see (\ref{eq43}) and recall the extremality of $u_*\in {\rm int}\,C_+$ and of $v_*\in-{\rm int}\,C_+$). Also, without any loss of generality, we may assume that
	$$\beta(v_*)\leq\beta(u_*).$$
	
	The reasoning is similar if the opposite inequality holds. Because of Claim \ref{cl2}, we can find $\rho\in(0,1)$ small such that
	\begin{equation}\label{eq48}
		\beta(v_*)\leq\beta(u_*)<\inf\{\beta(u):||u-u_*||=\rho\}=m_{\rho},\ ||v_*-u_*||>\rho
	\end{equation}
	(see Aizicovici, Papageorgiou \& Staicu \cite{1}, proof of Proposition 29). Since $\beta(\cdot)$ is coercive (see (\ref{eq43})), it satisfies the $PS$-condition. This fact and (\ref{eq48}) permit the use of Theorem \ref{th1} (the mountain pass theorem). So, we can find $y_0\in W^{1,p}_{0}(\Omega)$ such that
	\begin{equation}\label{eq49}
		m_{\rho}\leq\beta(y_0)\ \mbox{and}\ y_0\in K_{\beta}\subseteq[v_*,u_*]\ (\mbox{see Claim \ref{cl1}}).
	\end{equation}
	
	From (\ref{eq48}) and (\ref{eq49}), it follows that
	$$y_0\notin\{u_*,v_*\}.$$
	
	So, if we can show that $y_0\neq 0$, then $y_0$ will be nodal (see (\ref{eq49})). By the mountain pass theorem (see Theorem \ref{th1}), we have
	\begin{equation}\label{eq50}
		\beta(y_0)=\inf\limits_{\gamma\in\Gamma}\max\limits_{0\leq t\leq 1}\beta(\gamma(t)),
	\end{equation}
	with $\Gamma=\{\gamma\in C([0,1],W^{1,p}_{0}(\Omega)):\gamma(0)=v_*,\gamma(1)=u_*\}$. According to (\ref{eq50}), in order to show the nontriviality of $y_0$, it suffices to construct a path $\gamma_*\in \Gamma$ such that $\beta|_{\gamma_*}<0=\beta(0)$.
	
	To this end note that hypothesis $H_2(iii)$ implies that we can find $\xi_1\in(\hat{\lambda}_2,\xi_0)$ and $\delta>0$ such that
	\begin{equation}\label{eq51}
		F(z,x)\geq\frac{1}{p}\xi_1|x|^p\ \mbox{for almost all}\ z\in\Omega,\ \mbox{and all}\ |x|\leq\delta.
	\end{equation}
	
	Let $$\partial B^{L^p}_{1}=\{u\in L^p(\Omega):||u||_p=1\},\quad \widehat{M}=W^{1,p}_{0}(\Omega)\cap\partial B^{L^p}_{1}\quad \mbox{and}\quad \widehat{M}_c=\widehat{M}\cap C^1_0(\overline{\Omega}).$$
We introduce the following sets of path
	\begin{eqnarray*}
		&&\hat{\Gamma}=\{\hat{\gamma}\in C([-1,1],\widehat{M}):\hat{\gamma}(-1)=-\hat{u}_1,\hat{\gamma}(1)=\hat{u}_1\},\\
		&&\hat{\Gamma}_c=\{\hat{\gamma}\in C([-1,1],\widehat{M}_c):\hat{\gamma}(-1)=-\hat{u}_1,\hat{\gamma}(1)=\hat{u}_1\}.
	\end{eqnarray*}
	\begin{claim}\label{cl5}
		$\hat{\Gamma}_c$ is dense in $\hat{\Gamma}$ for the $C([-1,1],W^{1,p}_{0}(\Omega))$-topology.
	\end{claim}
	
	Let $\hat{\gamma}\in\hat{\Gamma}$ and for every $n\geq 1$ we consider the multifunction $T_n:[-1,1]\rightarrow 2^{C^1_0(\overline{\Omega})}$ defined by
	$$T_n(t)=\left\{\begin{array}{ll}
		\{-\hat{u}_1\}&\mbox{if}\ t=-1\\
		\{u\in C^1_0(\overline{\Omega}):||u-\hat{\gamma}(t)||<\frac{1}{n}\}&\mbox{if}\ t\in(-1,1)\\
		\{\hat{u}_1\}&\mbox{if}\ t=1.
	\end{array}\right.$$
	
	Evidently, $T_n(\cdot)$ has nonempty convex values, which are open sets if $t\in(-1,1)$. Also, from Papageorgiou \& Kyritsi \cite[pp. 458-463]{17}, we have that $T_n(\cdot)$ is a lower semicontinuous multifunction. So, we can apply Theorem 3.1''' of Michael \cite{16} (see also Hu \& Papageorgiou \cite[p. 97]{12}) and find a continuous map $\tau_n:[-1,1]\rightarrow C^1_0(\overline{\Omega})$ such that $\tau_n(t)\in T_n(t)$ for all $t\in[-1,1]$, all $n\geq 1$. We have
	\begin{eqnarray}
		&&||\tau_n(t)-\hat{\gamma}(t)||<\frac{1}{n}\ \mbox{for all}\ t\in[-1,1],\ \mbox{all}\ n\geq 1,\label{eq52}\\
		&\Rightarrow&||\tau_n(t)||_p\rightarrow||\hat{\gamma}(t)||_p\ \mbox{uniformly in}\ t\in[-1,1]\ \mbox{as}\ n\rightarrow\infty .\label{eq53}
	\end{eqnarray}
	
	So, for $n\geq 1$ big enough, we can define
	$$\hat{\gamma}_n(t)=\frac{\tau_n(t)}{||\tau_n(t)||_p}\ \mbox{for all}\ t\in[-1,1]\ (\mbox{recall}\ \hat{\gamma}(t)\in\partial B^{L^p}_{1}\ \mbox{for all}\ t\in[-1,1])$$
	
	Then we have
	\begin{eqnarray}\label{eq54}
		 ||\hat{\gamma}_n(t)-\hat{\gamma}(t)||&\leq&||\hat{\gamma}_n(t)-\tau_n(t)||+||\tau_n(t)-\hat{\gamma}(t)||\nonumber\\
		&\leq&|1-||\tau_n(t)||_p|\frac{||\tau_n(t)||}{||\tau_n(t)||_p}+\frac{1}{n}\ \mbox{for all}\ t\in[-1,1],\  n\geq 1\ \mbox{(see (\ref{eq52}))}.
	\end{eqnarray}
	
	Also since $||\hat{\gamma}(t)||_p=1$ for all $t\in[-1,1]$, we can write
	\begin{eqnarray}\label{eq55}
		&&|1-||\tau_n(t)||_p|=|||\hat{\gamma}(t)||_p-||\tau_n(t)||p|\nonumber\\
		&&\hspace{2.3cm}\leq||\hat{\gamma}(t)-\tau_n(t)||_p\nonumber\\
		&&\hspace{2.3cm}\leq c_4||\hat{\gamma}(t)-\tau_n(t)||\ \mbox{for some}\ c_4>0,\ \mbox{and all}\ t\in[-1,1],\ n\geq 1,\nonumber\\
		&\Rightarrow&\max\limits_{-1\leq t\leq 1}|1-||\tau_n(t)||_p|\leq c_4\frac{1}{n}\ \mbox{for all}\ n\geq 1\ (\mbox{see (\ref{eq52})}).
	\end{eqnarray}
	
	Returning to (\ref{eq54}) and using (\ref{eq55}), we obtain
	$$\max\limits_{-1\leq t\leq 1}||\hat{\gamma}_n(t)-\hat{\gamma}(t)||\rightarrow 0\ \mbox{as}\ n\rightarrow\infty.$$
	
	Evidently, $\hat{\gamma}_n\in\hat{\Gamma}_c$ for all $n\geq 1$. So, we have proved Claim \ref{cl5}.
	
	Using Claim \ref{cl5} and Proposition \ref{prop3}, given $\eta\in(0,\xi_1-\hat{\lambda}_2)$, we can find $\bar{\gamma}_0\in\hat{\Gamma}_c$ such that
	\begin{equation}\label{eq56}
		||D\bar{\gamma}_0(t)||^p_p\leq\hat{\lambda}_2+\eta\ \mbox{for all}\ t\in[-1,1].
	\end{equation}
	
	The set $\bar{\gamma}_0([-1,1])$ is compact in $C^1_0(\overline{\Omega})$. Also, $u_*\in {\rm int}\,C_+$ and $v_*\in-{\rm int}\,C_+$ (see Proposition \ref{prop7}). So, using also Lemma 3.3 of Filippakis, Kristaly \& Papageorgiou \cite{8}, we can find $\vartheta\in(0,1)$ small such that
	\begin{eqnarray}\label{eq57}
		&&\vartheta\bar{\gamma}_0(t)\in[v_*,u_*]\ \mbox{for all}\ t\in[-1,1]\ \mbox{and}\ |\vartheta\bar{\gamma}_0(t)(z)|\leq\delta\ \mbox{for all}\ t\in[-1,1],\ z\in\overline{\Omega}\\
		&&\hspace{12cm}(\mbox{see (\ref{eq51})}).\nonumber
	\end{eqnarray}
	
	Let $\hat{\gamma}_0=\vartheta\bar{\gamma}_0$. Then $\hat{\gamma}_0$ is a path in $W^{1,p}_{0}(\Omega)$ connecting $-\vartheta\hat{u}_1$ and $\vartheta\hat{u}_1$ and also we have
	\begin{eqnarray}\label{eq58}
		 &&\beta(\hat{\gamma}_0(t))=\frac{1}{p}||D\hat{\gamma}_0(t)||^p_p-\int_{\Omega}F(z,\hat{\gamma}_0(t))dz\ (\mbox{see (\ref{eq43}) and (\ref{eq57})})\nonumber\\
		&&\hspace{1.6cm}\leq\frac{1}{p}[\hat{\lambda}_2+\eta-\xi_1]||\hat{\gamma}_0(t)||^p_p\ \mbox{for all}\ t\in[-1,1]\ (\mbox{see (\ref{eq51}), (\ref{eq56}), (\ref{eq57})})\nonumber\\
		&&\hspace{1.6cm}<0\ \mbox{for all}\ t\in[-1,1]\ (\mbox{recall that}\ 0<\eta<\xi_1-\hat{\lambda}_2),\nonumber\\
		&\Rightarrow&\beta|_{\hat{\gamma}_0}<0.
	\end{eqnarray}
	
	Next, we produce a path in $W^{1,p}_{0}(\Omega)$ connecting $\vartheta\hat{u}_1$ and $u_*$ and along which $\beta$ is negative.
	
	To this end, let $a=\beta_+(u_*)$. From the proof of Claim \ref{cl4}, we know that $a<0$ and because of Claim \ref{cl3}, we see that
	\begin{equation}\label{eq59}
		K^{a}_{\beta_+}=\{u_*\}.
	\end{equation}
	
	Applying the second deformation theorem (see, for example, Gasinski \& Papageorgiou \cite[p. 628]{10}), we can find a deformation $h:[0,1]\times(\beta^0_+\backslash\{0\})\rightarrow\beta^0_+$ such that
	\begin{eqnarray}
		&&h(0,u)=u\ \mbox{for all}\ u\in \beta^0_+\backslash\{0\},\label{eq60}\\
		&&h(1,u)=u_*\ \mbox{for all}\ u\in \beta^0_+\backslash\{0\}\ (\mbox{see (\ref{eq59})}),\label{eq61}\\
		&&\beta_+(h(t,u))\leq\beta_+(h,(s,u))\ \mbox{for all}\ t,s\in[0,1],s<t,\ \mbox{all}\ u\in\beta^0_+\backslash\{0\}.\label{eq62}
	\end{eqnarray}
	
	We define
	$$\hat{\gamma}_+(t)=h(t,\vartheta\hat{u}_1)^+\ \mbox{for all}\ t\in[0,1].$$
	
	Evidently, this is a path in $W^{1,p}_{0}(\Omega)$ and
	\begin{eqnarray*}
		&&\hat{\gamma}_+(0)=\vartheta\hat{u}_1\ (\mbox{see (\ref{eq60}) and recall}\ \vartheta\hat{u}_1\in {\rm int}\,C_+),\\
		&&\hat{\gamma}_+(1)=u_*\ (\mbox{see (\ref{eq61}) and recall}\ u_*\in {\rm int}\,C_+).
	\end{eqnarray*}
	
	Also, since $\hat{\gamma}_+(t)(z)\geq 0$ for all $z\in\overline{\Omega}$, all $t\in[0,1]$, we have
	\begin{eqnarray}\label{eq63}
		 &&\beta(\hat{\gamma}_+(t))=\beta_+(\hat{\gamma}(t))\leq\beta_+(\vartheta\hat{u}_1)=\beta(\vartheta\hat{u}_1)<0\ \mbox{for all}\ t\in[0,1]\nonumber\\
		&&\hspace{7cm}(\mbox{see (\ref{eq58}) and (\ref{eq62})}),\nonumber\\
		&\Rightarrow&\beta|_{\hat{\gamma}_+}<0.
	\end{eqnarray}
	
	In a similar way, we can produce a path $\hat{\gamma}_-$ in $W^{1,p}_{0}(\Omega)$ which connects $-\vartheta\hat{u}_1$ and $v_*$ and such that
	\begin{equation}\label{eq64}
		\beta|_{\hat{\gamma}_-}<0.
	\end{equation}
	
	We concatenate $\hat{\gamma}_-,\hat{\gamma}_0,\hat{\gamma}_+$ and generate a path $\gamma_*\in\Gamma$ such that
	\begin{eqnarray*}
		&&\beta|_{\gamma_*}<0\ (\mbox{see (\eqref{eq58}, \eqref{eq63}, \eqref{eq64})},\\
		&\Rightarrow&y_0\neq 0,\\
		&\Rightarrow&y_0\in C^1_0(\overline{\Omega})\ \mbox{(nonlinear regularity) is a nodal solution of (\ref{eq1})}.
	\end{eqnarray*}
\end{proof}

So, we can state our first multiplicity theorem.
\begin{theorem}\label{th9}
	If hypotheses $H_2$ hold, then problem (\ref{eq1}) admits at least three nontrivial solutions
	$$u_0\in {\rm int}\,C_+,\ v_0\in-{\rm int}\,C_+\ \mbox{and}\ y_0\in[v_0,u_0]\cap C^1_0(\overline{\Omega})\ \mbox{nodal}.$$
\end{theorem}

Next, we change the geometry near the origin, by introducing a concave term. So, now the hypotheses on the reaction $f(z,x)$ are the following:

\smallskip
$H_3:$ $f:\Omega\times\RR\rightarrow\RR$ is a measurable function such that for almost all $z\in\Omega$, $f(z,0)=0,\ f(z,\cdot)$ is locally $\alpha$-H\"{o}lder continuous with $\alpha\in\left(0,1\right]$ and local H\"{o}lder constant $k\in L^{\infty}(\Omega)_+$ and
\begin{itemize}
	\item[(i)] for every $\rho>0$, there exists $a_{\rho}\in L^{\infty}(\Omega)_+$ such that
	$$|f(z,x)|\leq a_{\rho}(z)\ \mbox{for almost all}\ z\in\Omega,\ \mbox{and all}\ |x|\leq\rho;$$
	\item[(ii)] $\limsup\limits_{x\rightarrow\pm\infty}\frac{f(z,x)}{|x|^{p-2}x}\leq\xi<\hat{\lambda}_1$ uniformly for almost all $z\in\Omega$;
	\item[(iii)] there exist $q\in(1,p)$ and $\delta>0$ such that
	\begin{eqnarray*}
		&&0<f(z,x)x\leq qF(z,x)\ \mbox{for almost all}\ z\in\Omega,\ \mbox{and all}\ 0<|x|\leq\delta,\\
		&&0<\mbox{ess}\,\inf\limits_{\Omega}F(\cdot,\pm\delta),
	\end{eqnarray*}
	where $F(z,x)=\int^{x}_{0}f(z,s)ds$;\\
\item[(iv)] there exists $M_0>0$ such that for almost all $z\in\Omega$,
\begin{eqnarray*}
		&& x\mapsto\frac{f(z,x)}{x^{p-1}}\ \mbox{is nondecreasing on $[M_0,+\infty)$};\\
		&&x\mapsto\frac{f(z,x)}{|x|^{p-2}x}\ \mbox{is nonincreasing on $(-\infty, -M_0]$}.
	\end{eqnarray*}
\end{itemize}
\begin{remark}
	For example, we can think of a reaction of the form
	$$f(z,x)=k_0(z)|x|^{q-2}x+f_0(z,x),$$
	with $1<q<2,\ k_0\in L^{\infty}(\Omega)$ and $f_0(z,x)$ is a measurable function such that for almost all $z\in\Omega,\ f_0(z,\cdot)$ is locally $\alpha$-H\"{o}lder continuous with $\alpha\in(0,1)$ and local H\"{o}lder constant $k\in L^{\infty}(\Omega)_+$ and
	 $$\limsup\limits_{x\rightarrow\pm\infty}\frac{f_0(z,x)}{|x|^{p-2}x}\leq\xi_1<\hat{\lambda}_1\ \mbox{and}\ \lim\limits_{x\rightarrow 0}\frac{f_0(z,x)}{|x|^{p-2}x}=0\ \mbox{uniformly for almost all}\ z\in\Omega.$$
\end{remark}

We are ready to state and prove our second multiplicity theorem.
\begin{theorem}\label{th10}
	If hypotheses $H_3$ hold, then problem (\ref{eq1}) admits at least three nontrivial solutions
	$$u_0\in {\rm int}\,C_+,\ v_0\in-{\rm int}\,C_+\ \mbox{and}\ y_0\in[v_0,u_0]\cap C^1_0(\overline{\Omega})\ \mbox{nodal}.$$
\end{theorem}
\begin{proof}
	The two constant sign solutions come from Proposition \ref{prop6}.
	
	Let $u_*\in {\rm int}\,C_+$ and $v_*\in-{\rm int}\,C_+$ be the two extremal constant sign solutions produced in Proposition \ref{prop7}. Using them and reasoning as in the first part of the proof of Proposition \ref{prop8}, via the functional $\beta$ and the mountain pass theorem (see Theorem \ref{th1}), we obtain a third solution
	$$y_0\in[v_*,u_*]\cap C^1_0(\overline{\Omega}).$$
	
	Since $y_0$ is a critical point of mountain pass type for the functional $\beta$, we have
	\begin{equation}\label{eq65}
		C_1(\beta,y_0)\neq 0.
	\end{equation}
	
	On the other hand it is well-known that hypothesis $H_3(iii)$ implies that
	\begin{equation}\label{eq66}
		C_k(\beta,0)=0\ \mbox{for all}\ k\geq 0.
	\end{equation}
	
	Comparing (\ref{eq65}) and (\ref{eq66}) we infer that $y_0\not=0$. This means that $y_0\in[v_*,u_*]\cap C^1_0(\overline{\Omega})$ is a nodal solution of problem (\ref{eq1}).
\end{proof}

\section{The semilinear equation $(p=2)$}

In this section, we focus on the semilinear equation (that is, $p=2$). So, the problem under consideration is the following:
\begin{equation}\label{eq67}
	-\Delta u(z)=f(z,u(z))\ \mbox{in}\ \Omega,\ u|_{\partial\Omega}=0.
\end{equation}

By improving the regularity on the reaction $f(z,\cdot)$, we can produce a second nodal solution for a total of four nontrivial solutions for problem (\ref{eq67}).

The hypotheses on the reaction $f(z,x)$ are the following:

\smallskip
$H_4:$ $f:\Omega\times\RR\rightarrow\RR$ is a measurable function such that for almost all $z\in\Omega\ f(z,0)=0$, $f(z,\cdot)\in C^1(\RR)$ and
\begin{itemize}
	\item[(i)] for every $\rho>0$, there exists $a_{\rho}\in L^{\infty}(\Omega)$ such that
	$$|f'_x(z,x)|\leq a_{\rho}(z)\ \mbox{for almost all}\ z \in\Omega,\ \mbox{and all}\ |x|\leq\rho;$$
	\item[(ii)] $\limsup\limits_{x\rightarrow\pm\infty}\frac{f(z,x)}{x}\leq\xi<\hat{\lambda}_1$ uniformly for almost all $z\in\Omega$;
	\item[(iii)] $f'_x(z,0)=\lim\limits_{x\rightarrow 0}\frac{f(z,x)}{x}$ uniformly for almost all $z\in\Omega$ and there exists integer $m\geq 2$ such that
	$$\hat{\lambda}_m\leq f'_x(z,0)\leq\hat{\lambda}_{m+1}\ \mbox{for almost all}\ z\in\Omega$$
	with the first inequality being strict on a set of positive measure and for $F(z,x)=\int^{x}_{0}f(z,s)ds$ we have
	$$F(z,x)\leq\frac{\hat{\lambda}_{m+1}}{2}x^2\ \mbox{for almost all}\ z\in\Omega,\ \mbox{and all}\ x\in\RR;$$
\item[(iv)] there exists $M_0>0$ such that for almost all $z\in\Omega$,
\begin{eqnarray*}
		&& x\mapsto\frac{f(z,x)}{x^{p-1}}\ \mbox{is nondecreasing on $[M_0,+\infty)$};\\
		&&x\mapsto\frac{f(z,x)}{|x|^{p-2}x}\ \mbox{is nonincreasing on $(-\infty, -M_0]$}.
	\end{eqnarray*}
\end{itemize}
\begin{remark}
	The differentiability of $f(z,\cdot)$ and hypothesis $H_4(i)$ imply that $f(z,\cdot)$ is locally Lipschitz with locally Lipschitz constant in $L^{\infty}(\Omega)_+$.
\end{remark}

From Proposition \ref{prop7}, we know that we have extremal constant sign solutions
$$u_*\in {\rm int}\,C_+\ \mbox{and}\ v_*\in-{\rm int}\,C_+.$$

Using these extremal constant sign solutions, we consider the functional $\beta:H^1_0(\Omega)\rightarrow\RR$ introduced in the proof of Proposition \ref{prop8} (now $p=2$). We have $\beta\in C^{2-0}(H^1_0(\Omega))$ (that is $\beta$ is in $C^1(H^1_0(\Omega))$ with locally Lipschitz derivative).
\begin{prop}\label{prop11}
	If hypotheses $H_4$ hold, then $C_k(\beta,0)=\delta_{k,d_m}\ZZ$ for all $k\geq 0$ with $d_m=dim\overset{m}{\underset{\mathrm{k=1}}\oplus}E(\hat{\lambda}_k)\geq 2$.
\end{prop}
\begin{proof}
	If in hypothesis $H_4(iii)$ the inequality $f'_x(z,0)\leq\hat{\lambda}_{m+1}$ is also strict on a set (not necessarily the same) of positive measure, then $u=0$ is a nondegenerate critical point of $\beta$ and so from Li, Li \& Liu \cite{13} we have
	\begin{equation}\label{eq68}
		C_k(\beta,0)=\delta_{k,d_m}\ZZ\ \mbox{for all}\ k\geq 0.
	\end{equation}
	
	So, suppose that $f'_x(z,0)=\hat{\lambda}_{m+1}$ for almost all $z\in\Omega$. Using hypothesis $H_4(iii)$ and (\ref{eq5}), we have
	\begin{equation}\label{eq69}
		\beta(u)\geq\frac{1}{2}||Du||^2_2-\frac{\hat{\lambda}_{m+1}}{2}||u||^2\geq 0\ \mbox{for all}\ u\in\hat{H}_{m+1}=\overline{{\underset{{k\geq m+1}}\oplus}E(\hat{\lambda}_k)}.
	\end{equation}
	
	On the other hand, given $\epsilon>0$, we can find $\delta=\delta(\epsilon)>0$ such that
	\begin{equation}\label{eq70}
		F(z,x)\geq\frac{1}{2}(f'_x(z,0)-\epsilon)x^2\ \mbox{for almost all}\ z\in\Omega,\ \mbox{and all}\ x\in[-\delta,\delta].
	\end{equation}
	
	Since $\bar{H}_m=\overset{m}{\underset{\mathrm{k=1}}\oplus}E(\hat{\lambda}_k)$ is finite-dimensional, all norms are equivalent and so we can find small enough $\rho>0$ such that if $\bar{B}_{\rho}=\{u\in H^1_0(\Omega):||u||\leq\rho\}$, then
	\begin{equation}\label{eq71}
		u\in\bar{H}_m\cap B_{\rho}\Rightarrow|u(z)|\leq\delta\ \mbox{for all}\ z\in\overline{\Omega}\ \mbox{and}\ u\in[v_*,u_*].
	\end{equation}
	
	Let $u\in\bar{H}_m\cap B_{\rho}$. Then we have
	\begin{eqnarray*}
		 \beta(u)&\leq&\frac{1}{2}||Du||^2_2-\frac{1}{2}\int_{\Omega}f'_x(z,0)u^2dz+\frac{\epsilon}{2}||u||^2_2\ (\mbox{see (\ref{eq43}), (\ref{eq70}), (\ref{eq71})})\\
		&\leq&-\frac{c_5-\epsilon}{2}||u||^2\ \mbox{for some}\ c_5>0\ (\mbox{see Proposition \ref{prop4}}).
	\end{eqnarray*}
	
	Choosing $\epsilon\in(0,c_5)$, we infer that
	\begin{equation}\label{eq72}
		\beta(u)\leq 0\ \mbox{for all}\ u\in\bar{H}_m\cap\bar{B}_{\rho}.
	\end{equation}
	
	From (\ref{eq69}) and (\ref{eq72}) we see that $\beta$ has local linking at the origin and of course it is locally Lipschitz there. Therefore
	$$C_{d_m}(\beta,0)\neq 0.$$
	
	Invoking the shifting theorem for $C^{2-0}$ functionals due to Li, Li \& Liu \cite{13}, we conclude that
	$$C_k(\beta,0)=\delta_{k,d_m}\ZZ\ \mbox{for all}\ k\geq 0.$$
The proof is now complete.
\end{proof}

Now we are ready for our third multiplicity theorem concerning problem (\ref{eq67}).
\begin{theorem}\label{th12}
	If hypotheses $H_4$ hold, then problem (\ref{eq67}) admits at least four nontrivial solutions
	$$u_0\in {\rm int}\,C_+,\ v_0\in-{\rm int}\,C_+\ \mbox{and}\ y_0,\hat{y}\in int_{C^1_0(\overline{\Omega})}[v_0,u_0]\ \mbox{nodal.}$$
\end{theorem}
\begin{proof}
	From Proposition \ref{prop6}, we already have two nontrivial constant sign solutions
	$$u_0\in {\rm int}\,C_+\ \mbox{and}\ v_0\in-{\rm int}\,C_+.$$
	
	Moreover, by virtue of Proposition \ref{prop7} we may assume that $u_0$ and $v_0$ are extremal (that is, $u_0=u_*\in {\rm int}\,C_+$ and $v_0=v_*\in-{\rm int}\,C_+$). The differentiability of $f(z,\cdot)$ and hypothesis $H_4(i)$ imply that, if $\rho=\max\{||\bar{u}||_{\infty},||\bar{v}||_{\infty}\}$, then we can find $\hat{\xi}_{\rho}>0$ such that for almost all $z\in\Omega$ $x\rightarrow f(z,x)+\hat{\xi}_{\rho}x$ is nondecreasing on $[-\rho,\rho]$.
	
	As in the proof of Proposition \ref{prop8}, using the functional $\beta\in C^{2-0}(H^1_0(\Omega))$ and the mountain pass theorem (see Theorem \ref{th1}), we can find $y_0\in[v_0,u_0]\cap C^1_0(\overline{\Omega})$, which is a solution of problem (\ref{eq67}). We have
	\begin{eqnarray*}
		&&-\Delta y_0(z)+\hat{\xi}_{\rho}y_0(z)=f(z,y_0(z))+\hat{\xi}_{\rho}y_0(z)\\
		&&\hspace{3cm}\leq f(z,u_0(z))+\hat{\xi}_{\rho}u_0(z)\ (\mbox{since}\ y_0\leq u_0)\\
		&&\hspace{3cm}=-\Delta u_0(z)+\hat{\xi}_{\rho}u_0(z)\ \mbox{for almost all}\ z\in\Omega,\\
		&\Rightarrow&\Delta(u_0-y_0)(z)\leq\hat{\xi}_{\rho}(u_0-y_0)(z)\ \mbox{for almost all}\ z\in\Omega,\\
		&\Rightarrow&u_0-y_0\in {\rm int}\,C_+\ (\mbox{by the strong maximum principle}).
	\end{eqnarray*}
	
	Similarly, we show that $y_0-v_0\in {\rm int}\,C_+$. Therefore
	$$y_0\in \mbox{int}_{C^1_0(\overline{\Omega})}[v_0,u_0].$$
	
	Since $y_0$ is a critical point of mountain pass-type for $\beta$, we have from Theorem 2.7 of Li, Li \& Liu \cite{13}
	\begin{equation}\label{eq73}
		C_k(\beta,y_0)=\delta_{k,1}\ZZ\ \mbox{for all}\ k\geq 0.
	\end{equation}
	
	From Proposition \ref{prop11} we know that
	\begin{equation}\label{eq74}
		C_k(\beta,0)=\delta_{k,d_m}\ZZ\ \mbox{for all}\ k\geq 0\ \mbox{with}\ d_m\geq 2.
	\end{equation}
	
	Comparing (\ref{eq73}) and (\ref{eq74}), we infer that
	$$y_0\neq 0\ \mbox{and so}\ y_0\in \mbox{int}_{C^1_0(\overline{\Omega})}[v_0,u_0]\ \mbox{is a nodal solution of (\ref{eq67})}.$$
	
	Recall that $u_0,v_0$ are local minimizers of $\beta$ (see Claim \ref{cl4} in the proof of Proposition \ref{prop8}). Hence we have
	\begin{equation}\label{eq75}
		C_k(\beta,u_n)=C_k(\beta,v_0)=\delta_{k,0}\ZZ\ \mbox{for all}\ k\geq 0.
	\end{equation}
	
	Moreover, the coercivity of $\beta(\cdot)$ (see (\ref{eq43})), implies that
	\begin{equation}\label{eq76}
		C_k(\beta,\infty)=\delta_{k,0}\ZZ\ \mbox{for all}\ k\geq 0.
	\end{equation}
	
	Suppose that $K_{\beta}=\{0,u_0,v_0,y_0\}$. Then from (\ref{eq73})$\rightarrow$(\ref{eq76}) and the Morse relation (\ref{eq7}) with $t=-1$, we have
	\begin{eqnarray*}
		&&(-1)^{d_m}+2(-1)^0+(-1)^1=(-1)^0,\\
		&\Rightarrow&(-1)^{d_m}=0,\ \mbox{a contradiction}.
	\end{eqnarray*}
	
	So, there exists $\hat{y}\in K_{\beta},\ \hat{y}\notin\{0,u_0,v_0,y_0\}$. Then $\hat{y}\in[v_0,u_0]\cap C^1_0(\overline{\Omega})$ is nodal (see Claim \ref{cl3} in the proof of Proposition \ref{prop8} and use standard regularity theorem). In fact, as we did in the beginning of the proof for $y_0$, we can show that
	$$\hat{y}\in \mbox{int}_{C^1_0(\overline{\Omega})}[v_0,u_0].$$
The proof is now complete.
\end{proof}

\section{A special case}

In this section, we consider a special case of problem (\ref{eq1}) under hypotheses $H_2$, which we encounter in the literature.

So, we deal with the following parametric nonlinear Dirichlet problem
\begin{equation}\label{eq77}
	-\Delta_pu(z)=\lambda|u(z)|^{p-2}u(z)-g(z,u(z))\ \mbox{in}\ \Omega,\ u|_{\partial\Omega}=0,\ \lambda>0.
\end{equation}

We impose the following conditions on the perturbation $g(z,x)$.

\smallskip
$H_5:$ $g:\Omega\times\RR\rightarrow\RR$ is a measurable function such that for almost all $z\in\Omega$, $g(z,0)=0$, $g(z,\cdot)$ is locally $\alpha$-H\"{o}lder continuous with $\alpha\in\left(0,1\right]$ and local H\"{o}lder constant $k\in L^{\infty}(\Omega)_+$ and
\begin{itemize}
	\item[(i)] for every $\rho>0$, there exists $a_{\rho}\in L^{\infty}(\Omega)_+$ such that
	$$|g(z,x)|\leq a_{\rho}(z)\ \mbox{for almost all}\ z\in\Omega,\ \mbox{and all}\ |x|\leq\rho;$$
	\item[(ii)] $\liminf\limits_{x\rightarrow\pm\infty}\frac{g(z,x)}{|x|^{p-2}x}\geq\xi^*>\lambda-\hat{\lambda}_1$ uniformly for almost all $z\in\Omega$;
	\item[(iii)] there exist $\xi_0,\xi_*\in\RR,\ \xi_*<\lambda-\hat{\lambda}_2$ such that
	$$\xi_0\leq\liminf\limits_{x\rightarrow 0}\frac{g(z,x)}{|x|^{p-2}x}\leq\limsup\limits_{x\rightarrow 0}\frac{g(z,x)}{|x|^{p-2}x}\leq\xi_*\ \mbox{uniformly for almost all}\ z\in\Omega;$$
\item[(iv)] there exists $M_0>0$ such that for almost all $z\in\Omega$,
\begin{eqnarray*}
		&& x\mapsto\frac{g(z,x)}{x^{p-1}}\ \mbox{is nondecreasing on $[M_0,+\infty)$};\\
		&&x\mapsto\frac{g(z,x)}{|x|^{p-2}x}\ \mbox{is nonincreasing on $(-\infty, -M_0]$}.
	\end{eqnarray*}
\end{itemize}

Setting $f(z,x)=\lambda|x|^{p-2}x-g(z,x)$ and using Theorem \ref{th9}, we can state the following multiplicity theorem for problem (\ref{eq77}).
\begin{theorem}\label{th13}
	If hypotheses $H_5$ hold and $\lambda>\hat{\lambda}_2$ then problem (\ref{eq77}) admits at least three nontrivial solutions
	$$u_0\in {\rm int}\,C_+,\ v_0\in-{\rm int}\,C_+\ \mbox{and}\ y_0\in[v_0,u_0]\cap C^1_0(\overline{\Omega})\ \mbox{nodal}.$$
\end{theorem}
\begin{remark}
	This theorem complements the multiplicity result of Papageorgiou \& Papageorgiou \cite{18}.
\end{remark}

In the semilinear case $(p=2)$, we can say more. So, now the problem under consideration is the following:
\begin{equation}\label{eq78}
	-\Delta u(z)=\lambda u(z)-g(z,u(z))\ \mbox{in}\ \Omega,\ u|_{\partial\Omega}=0,\ \lambda>0.
\end{equation}

The hypotheses on the perturbation $g(z,x)$ are the following:

\smallskip
$H_6:$ $g:\Omega\times\RR\rightarrow\RR$ is a measurable function such that for almost all $z\in\Omega$, $g(z,0)=0$, $g(z,\cdot)\in C^1(\RR)$ and
\begin{itemize}
	\item[(i)] for every $\rho>0$, there exists $a_{\rho}\in L^{\infty}(\Omega)_+$ such that
	$$|g'_x(z,x)|\leq a_{\rho}(z)\ \mbox{for almost all}\ z\in\Omega,\ \mbox{and all}\ |x|\leq\rho;$$
	\item[(ii)] $\liminf\limits_{x\rightarrow\pm\infty}\frac{g(z,x)}{x}\geq\xi^*>\lambda-\hat{\lambda}_1$ uniformly for almost all $z\in\Omega$;
	\item[(iii)] $g'_x(z,0)=\lim\limits_{x\rightarrow 0}\frac{g(z,x)}{x}=0$ uniformly for almost all $z\in\Omega$;
\item[(iv)] there exists $M_0>0$ such that for almost all $z\in\Omega$,
\begin{eqnarray*}
		&& x\mapsto\frac{g(z,x)}{x^{p-1}}\ \mbox{is nondecreasing on $[M_0,+\infty)$};\\
		&&x\mapsto\frac{g(z,x)}{|x|^{p-2}x}\ \mbox{is nonincreasing on $(-\infty, -M_0]$}.
	\end{eqnarray*}
\end{itemize}

Again, we set $f(z,x)=\lambda x-g(z,x)$ and using Theorem \ref{th12}, we can state the following multiplicity theorem for problem (\ref{eq78}).
\begin{theorem}\label{th14}
	If hypotheses $H_6$ hold and $\lambda>\hat{\lambda}_2$, then problem (\ref{eq78}) has at least four nontrivial solutions
	$$u_0\in {\rm int}\,C_+,\ v_0\in-{\rm int}\,C_+\ \mbox{and}\ y_0,\hat{y}\in {\rm int}_{C^1_0(\overline{\Omega})}[v_0,u_0]\ \mbox{nodal}.$$
\end{theorem}
\begin{remark}
	This theorem complements the multiplicity results of Ambrosetti \& Lupo \cite{2}, Ambrosetti \& Mancini \cite{3} and Struwe \cite{19, 20}, which produce only three solutions and there are no nodal solutions among them.
\end{remark}

\medskip
{\bf Acknowledgements.} This research was supported by the Slovenian Research Agency grants
P1-0292, J1-8131, J1-7025, N1-0064, and N1-0083. V.D.~R\u adulescu acknowledges the support through a grant of the Romanian Ministry of Research and Innovation, CNCS--UEFISCDI, project number PN-III-P4-ID-PCE-2016-0130,
within PNCDI III.


\begin{thebibliography}{99}

\bibitem{1}   S. Aizicovici, N.S. Papageorgiou, and V. Staicu, {\it Degree theory for operators of monotone type and nonlinear elliptic equations with inequality constraints}, Memoirs Amer. Math. Soc., Vol. 196, No. \textbf{915} (November 2008).

\bibitem{2}   A. Ambrosetti and D. Lupo,  On a class of nonlinear Dirichlet problems with multiple solutions, {\it Nonlinear Anal.} \textbf{8} (1984), 1145-1150.

\bibitem{3}   A. Ambrosetti and G. Mancini,  Sharp nonuniqueness results for some nonlinear problems, {\it Nonlinear Anal.} \textbf{3} (1979), 635-645.

\bibitem{4}   A. Ambrosetti and P. Rabinowitz,  Dual variational methods in critical point theory and applications, {\it J. Functional Anal.} \textbf{14} (1973), 349-381.

\bibitem{5}   D. Arcoya and D. Ruiz, The Ambrosetti-Prodi problem for the $p$-Laplace operator, {\it Commun. Partial Diff. Equations} \textbf{31} (2006), 849-865.

\bibitem{6}   M. Cuesta, D. de Figueiredo, and J.-P. Gossez,  The beginning of the Fu\v{c}ik spectrum of the $p$-Laplacian, {\it J. Differential Equations} \textbf{159} (1999), 212-238.

\bibitem{7}   M. Filippakis, L. Gasinski, and N.S. Papageorgiou, Nonlinear periodic problems with nonsmooth potential restricted in one direction, {\it Publ. Math. Debrecen} \textbf{68} (2006), 37-62.

\bibitem{8}   M. Filippakis, A. Kristaly, and N.S. Papageorgiou, Existence of five nonzero solutions with exact sign for a $p$-Laplacian operator, {\it Discrete Cont. Dynam. Systems} \textbf{24} (2009), 405-440.

\bibitem{9}   J. Garcia Azorero, J. Manfredi, and J. Peral Alonso, Sobolev versus H\"{o}lder local minimizers and global multiplicity for some quasilinear elliptic equations, {\it Commun. Contemp. Math.} \textbf{2} (2000), 385-404.

\bibitem{10}   L. Gasinski and N.S. Papageorgiou, {\it Nonlinear Analysis}, Chapman \& Hall/CRC, Boca Raton, Fl., 2006.

\bibitem{11}   L. Gasinski and N.S. Papageorgiou, {\it Exercises in Analysis: Part I}, Springer, New York, 2014.

\bibitem{12}   S. Hu and N.S. Papageorgiou, {\it Handbook of Multivalued Analysis. Volume I: Theory}, Kluwer Academic Publishers, Dordrecht, The Netherlands, 1997.

\bibitem{13}   C. Li, S. Li, and J. Liu,  Splitting theorem, Poincar\'e-Hopf theorem and jumping nonlinear problems, {\it J.~Functional Anal.} \textbf{221} (2005), 439-455.

\bibitem{14}   J. Liu and S. Liu, The existence of multiple solutions to quasilinear elliptic equations, {\it Bull. London Math. Soc.} \textbf{37} (2005), 592-600.

\bibitem{15}   S. Liu,  Multiple solutions for coercive $p$-Laplacian equations, {\it J. Math. Anal. Appl.} \textbf{316} (2006), 229-236.

\bibitem{16}   E. Michael, Continuous selections $I$, {\it Annals Math.} \textbf{63} (1956), 361-382.

\bibitem{17}   N.S. Papageorgiou and S. Kyritsi, {\it Handbook of Applied Analysis}, Springer, New York, 2009.

\bibitem{18}   E.H. Papageorgiou and N.S. Papageorgiou, A multiplicity theorem for problems with the $p$-Laplacian, {\it J. Functional Anal.} \textbf{244} (2007), 63-77.

\bibitem{prr1}   N.S. Papageorgiou, V.D. R\u adulescu, and D.D. Repov\v{s}, Robin problems with a general potential and a superlinear reaction, {\it J. Differential Equations} {\bf 263} (2017), no. 6, 3244-3290.

\bibitem{prr2}   N.S. Papageorgiou, V.D. R\u adulescu, and D.D. Repov\v{s}, Positive solutions for super-diffusive mixed problems, {\it Appl. Math. Lett.} {\bf 77} (2018), 87-93.

\bibitem{19}   M. Struwe, A note on a result of Ambrosetti and Mancini, {\it Ann. Mat. Pura Appl.} \textbf{81} (1982), 107-115.

\bibitem{20}   M. Struwe, {\it Variational Methods}, Springer-Verlag, Berlin, 1990.

\bibitem{21}   S. Villegas, A Neumann problem with asymmetric nonlinearity and a related minimizing problem, {\it J. Differential Equations} \textbf{145} (1998), 145-155.

\end{thebibliography}
\end{document}